\documentclass[10pt,a4paper,english]{article}

\usepackage{hyperref}
\usepackage{setspace}
\usepackage{lipsum}
\usepackage{authblk}
\usepackage[style=alphabetic,sorting=nyt,backend=bibtex,firstinits=true]{biblatex}
\AtBeginBibliography{\setstretch{1}\footnotesize}
\AtEveryBibitem{%
  \iffieldundef{url}{}{\clearfield{doi}}%
}

\usepackage{dsfont}
\usepackage{amsfonts,amssymb,stmaryrd,amsmath,amsthm,amssymb}
\usepackage{graphicx,color}
\usepackage{multimedia}
\usepackage{float}
\usepackage{placeins}

\usepackage{algorithm}

\usepackage{mathtools}
\usepackage{mathrsfs}

\usepackage[a4paper]{geometry}
\newtheorem{theorem}{Theorem}[section]
\newtheorem{lemma}[theorem]{Lemma}
\newtheorem{remark}[theorem]{Remark}
\newtheorem{proposition}[theorem]{Proposition}

\newtheorem{definition}[theorem]{Definition}

\newcommand \p {\partial}

\newcommand \R {\mathbb{R}}
\newcommand \N {\mathbb{N}}
\renewcommand \L {\mathrm{L}}

\renewcommand \H {\mathrm{H}}

\newcommand \I {\mathrm{I}}
\newcommand \Id {\mathrm{Id}}

\renewcommand \d {\mathrm{d}}

\renewcommand \det {\mathrm{det}}
\newcommand \trace {\mathrm{trace}}

\DeclareMathOperator{\divg}{div}

\begingroup\makeatletter\ifx\SetFigFont\undefined%
\gdef\SetFigFont#1#2#3#4#5{%
  \reset@font\fontsize{#1}{#2pt}%
  \fontfamily{#3}\fontseries{#4}\fontshape{#5}%
  \selectfont}%
\fi\endgroup%

\title{Relaxation approach for learning neural network regularizers\\for a class of identification problems
}

\author[1, 2]{S\'ebastien Court}
\affil[1]{\begin{small}Department of Mathematics, University of Innsbruck, Technikerstrasse 13, 6020 Innsbruck, Austria.\end{small}}
\affil[2]{\begin{small}Digital Science Center, University of Innsbruck, Innrain 15, 6020 Innsbruck, Austria. Email: {\tt sebastien.court@uibk.ac.at}\end{small}}

\begin{document}

\maketitle

\begin{abstract}
The present paper deals with the data-driven design of regularizers in the form of artificial neural networks, for solving certain inverse problems formulated as optimal control problems. These regularizers aim at improving accuracy, wellposedness or compensating uncertainties for a given class of optimal control problems (inner-problems). Parameterized as neural networks, their weights are chosen in order to reduce a misfit between data and observations of the state solution of the inner- optimal control problems. Learning these weights constitutes the outer-problem.  Based on necessary first-order optimality conditions for the inner-problems, a relaxation approach is proposed in order to implement efficient solving of these inner-problems, namely the forward operator of the outer-problem. Optimality conditions are derived for the latter, and are implemented in numerical illustrations dealing with the inverse conductivity problem. The numerical tests show the feasibility of the relaxation approach, first for rediscovering standard $L^2$-regularizers, and next for designing regularizers that compensate unknown noise on the observed state of the inner-problem.
\end{abstract}

\hfill 

\tableofcontents

\section{Introduction}

Designing regularizers for solving inverse problems is an important data-science topic for which data-driven approaches have been recently developed. In particular, the design of regularizers as artificial neural networks demonstrate interesting capacities for improving characteristics of inverse problems. Among these characteristics, one can aim at recovering uniqueness of solutions for initially ill-posed inverse problems, facilitating their solving, obtaining specific qualitative properties like sparsity, or even compensating unknown real-world noise that leads to uncertainties on the measured data. The design of such regularizers can be achieved towards an optimal control approach, consisting in optimizing parameters of a {\it machine}, namely an input-output map, in order to minimize a distance between real-world data and outputs of this machine.

\subsection{General learning problem}

Consider a series of $K$ problems that belong to a given class of optimal control problems indexed by $1\leq k \leq K$, and assume that for each of them we know a solution. More precisely, we consider a set of tasks $\hat{y}_k$ that we know to be achieved by the means of parameters $\hat{u}_k$ (the {\it controls}). The $K$ pairs $\left\{(\hat{y}_k, \hat{u}_k), 1 \leq k \leq K\right\}$ constitutes our data set, and our main goal is, from this data set, to train a a function denoted by~$r$ that improves the solving of this class of control problem. Our supervised-learning problem can be formulated as follows:

\begin{equation} \label{pbmain} \tag{$\ast\ast$}
\left\{ \begin{array} {l}
\displaystyle \min_{r\in \mathcal{R}} \quad
\frac{1}{2K} \sum_{k=1}^K \|  u_k - \hat{u}_k \|_{\mathcal{U}}^2, \\[15pt]
\text{subject to: $u_k$ is solution of Problem~\eqref{pbstandard} with $(\hat{y}_k,r)$, for all $1\leq k\leq K$,}
\end{array} \right.
\end{equation}

where Problem~$\eqref{pbstandard}$ with $(\hat{y},r) = (\hat{y}_k,r)$ is given by

\begin{equation} \tag{$\ast$} \label{pbstandard}
\left\{ \begin{array} {l}
\displaystyle \min_{(y,u)\in \mathcal{Y}\times \mathcal{U}} \big(
c(y,\hat{y}) + \gamma \circ r(u)\big), \\[10pt]
\text{subject to: } \varphi(y,u)= 0.
\end{array} \right.
\end{equation}

Problems~\eqref{pbstandard} constitute the class of inner-problems for which we want to design a function~$r\in \mathcal{R}$ acting on the parameters~$u$ to identify. Designing this function~$r\in\mathcal{R}$ constitutes the outer-problem~\eqref{pbmain}.
The set $\mathcal{R}$ describes the class of functions in which we design a -- possible -- optimal function $r$. In Problem~\eqref{pbstandard}, the state is denoted by $y$, and the state equation is given by $\varphi(y,u) = 0$. The cost $c(y,\hat{y})$ corresponds to a misfit between the state $y$ and the task $\hat{y}$ to achieve. The mapping $\gamma$ is given, with values in $\R$, and can correspond for example to a norm.

\subsection{Motivation of data-driven regularization and state-of-the-art}

In the formulation of the inner-problem~\eqref{pbstandard}, the additional term $ r(u) $ corresponds to a so-called {\it regularizer}. Its role consists in forcing the parameter/control function $u$ to present certain properties, like regularity, boundedness, or sparsity for instance. Besides the qualitative properties aforementioned, designing a regularizer for a class of identification problems like~\eqref{pbstandard} can address several issues. For example, inverse problems are well-known to present non-unique solutions, leading to describe them as being {\it ill-posed}. A regularizer designed with a data-driven approach could help to select a relevant solution. Another issue could be due to uncertainty on measurements: Data can be subject to a noise phenomenon whose nature is unknown, and that would lead to inaccuracy on the measured solution. An appropriate regularizer, implementing implicitly the effects of this noise, could help to compensate this lack of accuracy. We illustrate this specific application in our numerical realizations (see section~\ref{sec-num}). 

In order to design such a regularizer, one need to parameterize the set $\mathcal{R}$, that we choose as a set of artificial neural networks (ANNs) with a prescribed number of layers. A regularizer $(u\mapsto r(w,u)) \in \mathcal{R}$ is then a mapping parameterized by a set of weights $w$. We refer to section~\ref{sec-NN} for further details on the definitions and notation related to the ANNs we will consider. Beside the theoretical approximation capacities of ANNs (see~\cite{Grohs2019} for example), data-driven methods training ANNs have lead to the design of regularizers that outperform classical solving methods for inverse problems.

For example, the data-driven modeling of regularizers with NETT (Network Tiknonov) was investigated in~\cite{NETT2018}, in order to solve inverse problems subject to unknown noise, with applications to Computer Tomography. Let us also cite~\cite{Obmann2019} in the context of trained decoder networks for sparse reconstruction, and refer to~\cite{Li2020} for a theoretical analysis of such an approach. Further developments were studied in~\cite{Obmann2019-2, Obmann2020}, considering aNETT (augmented Network Thikonov), designing a regularizer and also a penalty (misfit) function in the form of a neural network. Still on applications in medical imaging, the use of adversarial neural networks as regularizers was studied in~\cite{Lunz2018}. The use of deep neural networks for solving ill-posed inverse problems related to medical imaging was also studied in~\cite{Adler2017}. The special case of input-convex neural networks (ICNNs) for modeling regularizers with a data-driven approach was considered in~\cite{Mukherjee2020}. More generally, we orient the reader to~\cite{Benning2018, Haltmeier2020} for a comprehensive review on techniques related to this topic.


\subsection{Approach and strategy}

In~\cite{HKK2018} the authors investigated the problem of learning the Thikhonov regularization parameter for a class of inverse problems, with a {\it bi-level} approach. This approach consists in adopting a hierarchy in the formulation of the regularizer learning problem: The identification problems~\eqref{pbstandard} for which we want to design a regularizer are called {\it inner-problems}, and the main problem~\eqref{pbmain} consists in optimizing the parameters of the regularizer for a given set of identification problems for which we know the ground-truth (the {\it data set}). The term {\it optimal control problem} will be employed for both inner and outer problems. At a meta-level, the outer-problem consists in optimizing parameters (here the parameters~$w$ of the regularizer) of the family of the inner- control problems, for which the parameters~$u$ to identify are also described as {\it controls}. Let us mention that the bi-level approach was initially developed for image denoising in~\cite{PKK2013}.

In the present paper, we adopt this bi-level approach, where the inner-problem~\eqref{pbstandard} is considered as a constraint in the outer-problem. Compared with~\cite{HKK2018}, in our case the range of possible regularizers is {\it a priori} wider than a Thikonov regularizer, as the regularizers are chosen in the form of artificial neural networks parameterized with weights. The price is to pay is a theoretical analysis (existence and uniqueness of minimizers) that becomes out-of-reach in the general case, due in particular to the purely nonlinear aspects of neural networks. Therefore, our main contribution is the development of a relaxation approach, consisting in replacing the optimality constraint of the inner-problems by the corresponding necessary Karush-Kuhn-Tucker conditions. In our case these conditions are equality constraints (vanishing gradient equalities) that transform the outer-problem into a standard optimal control problem of state constraint type.

This relaxation approach is inspired by optimal control problems with partial differential equations constraints. Indeed, partial differential equations modeling physical dynamical systems can often be derived from the least action principle. They are merely first-order optimality conditions of a minimization problem. Wellposedness questions for these partial differential equations are then related to existence and uniqueness of minimizers for functionals called {\it actions}. 

The main theoretical result is given in Theorem~\ref{theorem-main}, where we provide the sensitivity of the objective function for the relaxed version of Problem~\eqref{pbmain}  
with respect to the weights of the neural network. This allows us to implement the determination of these weights in Algorithm~\ref{algo-outer}, with gradient rules. Illustrations are proposed by considering as inner-problems the classical identification of coefficients in  elliptic partial differential equations. Regularizers are designed, in particular for compensating the effects of a Gaussian noise introduced on the state variable.

\subsection{Plan}

The paper is organized as follows: Notation and assumptions are given in section~\ref{sec-not}. In section~\ref{sec-inner} we study the inner-problems with ANNs as regularizers. The latter are described in section~\ref{sec-NN}. The necessary optimality conditions related to the inner-problems are investigated in sections~\ref{sec-direct-app} and~\ref{sec-opt-inner}. The relaxation approach is presented in section~\ref{sec-relaxed-outer}. We prove existence of minimizers for the relaxed outer-problem in section~\ref{sec-exist-outer}, and derive optimality conditions in section~\ref{sec-opt-outer} for the latter. In section~\ref{sec-algo} we present algorithms for implementing the relaxation approach. Section~\ref{sec-num} is devoted to numerical illustrations considering the identification of coefficients in an elliptic partial differential equation. Concluding remarks are given in section~\ref{sec-conclusion}. In Appendix~\ref{appendix1} we present advantages that the optimal control approach can offer when dealing with state-constrained identification problems. Appendix~\ref{appendix2} recalls classical results related to the inverse conductivity problem.

\subsection*{Acknowledgments}
The author would like to thank Professor Markus Haltmeier for his
advises and helpful remarks on this  work.

\subsection*{Link to the code for numerical implementation}
The C++ code with which the numerical experiments were performed in section~\ref{sec-num} is available here:\\
\url{https://github.com/SebastienCourt/Learn_reg}

\section{Notation, assumptions and preliminaries} \label{sec-not}
The reader is invited to refer to the present section when browsing the rest of the paper, especially regarding the notation.

\subsection{Functional setting}
We denote by $\mathcal{Y}$, $\mathcal{U}$, $\mathcal{Q}$ real Banach spaces, and by $\varphi : \mathcal{Y} \times \mathcal{U} \rightarrow  \mathcal{Q}'$ the state mapping. We assume that $\mathcal{Q}$ is reflexive.
The control space $\mathcal{U}$ is assumed to be finite-dimensional, of dimension $n_{\mathrm{in}}$, so that $\mathcal{U} \simeq \R^{n_{\mathrm{in}}}$. Therefore in particular $\mathcal{U}$ is a Hilbert space, and reflexive. The range space $\mathcal{V}$ of the neural networks is finite-dimensional, of dimension $n_{\mathrm{out}}$, and $\mathcal{V} \simeq \R^{n_{\mathrm{out}}}$. The space of weights is denoted by $\mathcal{W}$, and described in section~\ref{sec-NN}. Given weights $w\in \mathcal{W}$, we denote by $r(w,\cdot):\mathcal{U} \rightarrow  \mathcal{V}$ the associated neural network of $\mathcal{R}$, and by $\gamma : \mathcal{V} \rightarrow \R$ a smooth given mapping, possibly chosen as a norm on $\R^{n_{\mathrm{out}}}$.

The data $\left\{(\hat{y}_k, \hat{u}_k), 1\leq k \leq K\right\}$ are of finite number. The controls $(\hat{u}_k)_{1\leq k \leq K}$ are assumed to lie in $\mathcal{U}$, while the tasks $(\hat{y}_k)_{1\leq k \leq K}$ lie in a finite-dimensional subspace of~$\mathcal{Y}$.

\subsection{The regularizer as realization of a neural network} \label{sec-NN0}
Given $L \geq 1$, we denote by $w = (w_1, w_2, \dots, w_L)$ a vector of affine functions. The set of regularizers, denoted by $\mathcal{R}$, is made of {\it feedforward neural networks} of $L$ layers, for which the dimensions $n_{\ell}$ ($1\leq \ell \leq L$) of the hidden layers are prescribed:
\begin{eqnarray*}
\mathcal{R} & = & \left\{ \mathcal{U} \ni u\mapsto
w_L\rho\left(w_{L-1}\rho\left(\left(\dots w_2\rho(w_1 u)\right)\right)\right) \in \mathcal{V}\mid \ 
w = (w_1, \dots , w_L) \in \prod_{\ell=1}^L \mathrm{Aff}(n_{\ell}, n_{\ell+1},\R)
\right\}
\end{eqnarray*}
We denoted by $\mathrm{Aff}(n_{\ell}, n_{\ell+1},\R)$ the set of affine functions from $\R^{n_{\ell}}$ to $\R^{n_{\ell+1}}$. For $1 \leq \ell \leq L$, its components write
\begin{equation*}
w_\ell: u \mapsto A_\ell u + b_\ell,
\end{equation*}
where $A_\ell \in \R^{n_{\ell+1}\times n_\ell}$ and $b_\ell \in \R^{n_{\ell+1}}$, with $n_1 = n_{\mathrm{in}} = \mathrm{dim}(\mathcal{U})$ and $n_{L+1} = n_{\mathrm{out}} =\mathrm{dim}(\mathcal{V})$. For the sake of concision, we will denote throughout the paper
\begin{equation*}
\mathcal{W}  :=  \prod_{\ell=1}^L \mathrm{Aff}(n_{\ell}, n_{\ell+1},\R).
\end{equation*}
We introduce $\mathcal{L}_{\mathcal{W}} := \displaystyle \sum_{\ell=1}^{L-1} n_{\ell}$. We endow the space $\mathcal{W}$ with the Euclidean norm given by
\begin{equation*}
\| w \|^2_{\mathcal{W}} := \sum_{\ell=1}^L 
\left(|A_{\ell}|^2_{\R^{n_{\ell+1}\times n_\ell}} + |b_\ell|^2_{\R^{n_{\ell+1}}} \right),
\end{equation*}
where $|A_{\ell}|^2_{\R^{n_{\ell+1}\times n_\ell}} = \trace(A_{\ell}^TA_{\ell})$ and $|b_{\ell}|$ is the Euclidean norm of $\R^{n_{\ell+1}}$. Given a smooth activation function $\rho : \R \rightarrow \R$ and a vector $w$ of $L$ weights, a {\it neural network} $r(w,\cdot) \in \mathcal{R}$ of $L$ layers writes as follows:
\begin{eqnarray*}
r(w,u) & = & 
\left\{ \begin{array} {ll}
w_2(\rho (w_{1}(u))) 
& \text{if } L=2 , \\[5pt]
w_L(\rho (w_{L-1}(\rho(\dots \rho(w_1(u)))))) & \text{if } L \geq 3.
\end{array} \right.
\end{eqnarray*}
In the expression above, since the functions $w_\ell$ have values in $\R^{n_{\ell+1}}$, the compositions by $\rho$ 
have to be understood coordinate-wise as $\rho(w_\ell(\mathrm{z}))_i = \rho((w_\ell(\mathrm{z}))_i)$, for $\mathrm{z} \in \R^{n_\ell}$ and $i \in \{1,\dots,n_{\ell+1}\}$.
Thus, the regularizer space $\mathcal{R}$ is parameterized by the weights lying in $\mathcal{W}$, and the outer-problem~\eqref{pbmain} is formulated as follows:
\begin{equation} \label{pbmainweights} \tag{$\ast\ast$}
\left\{ \begin{array} {l}
\displaystyle \min_{w\in \mathcal{W}} \quad
\frac{1}{2} \sum_{k=1}^K \| \hat{u}_k - u_k \|_{\mathcal{U}}^2, \\[15pt]
\text{subject to: $u_k$ is solution of Problem~\eqref{pbstandard} with $(\hat{y},r) = \big(\hat{y}_k,r(w,\cdot)\big)$, for all $1\leq k\leq K$.}
\end{array} \right.
\end{equation}
We keep the same notation for this new description of the outer-problem.

\subsection{General assumptions}
Throughout the paper, we will assume that the following set of assumptions holds:
\begin{description}
\item[$(\mathbf{A1})$] The mapping $\varphi$ is of class $C^2$ over $\mathcal{Y}\times \mathcal{U}$. For all $u\in \mathcal{U}$, there exists $y\in \mathcal{Y}$ such that $\varphi(y,u) = 0$.

\item[$(\mathbf{A2})$] For all $(\overline{y},\overline{u}) \in \mathcal{Y} \times \mathcal{U}$, the linear mapping $ \varphi'_y(\overline{y},\overline{u}) : \mathcal{Y} \rightarrow \mathcal{Q}'$ is surjective. Moreover, there exists a constant $C(\overline{y},\overline{u})>0$, depending only on $(\overline{y},\overline{u})$ such that for all $y\in \mathcal{Y}$ the following estimate holds:
\begin{eqnarray*}
 \| y\|_{\mathcal{Y}} & \leq & C \left\| \varphi'_y(\overline{y},\overline{u}).y \right\|_{\mathcal{Q}'} .
\end{eqnarray*}

Instead of~$(\mathbf{A2})$, we may consider a weaker assumption, namely~$(\mathbf{A2})'$:

\item[$(\mathbf{A2})'$] For all $(\overline{y},\overline{u}) \in \mathcal{Y} \times \mathcal{U}$, the linear mapping $ \varphi'_y(\overline{y},\overline{u}) : \mathcal{Y} \rightarrow \mathcal{Q}'$ is almost surjective, namely its range is dense in $\mathcal{Q}'$.

\item[$(\mathbf{A3})$] When $r \equiv 0$, Problem~\eqref{pbstandard} admits at least one global minimizer.

\item[$(\mathbf{A4})$] For all~$\hat{y}$ the cost function $c(\cdot,\hat{y}): \mathcal{Y} \rightarrow  \R$ is twice continuously Fr\'echet-differentiable.

\item[$(\mathbf{A5})$] The activation function $\rho$ is of class $C^2$ over $\R$.
\end{description}

Assumption~$(\mathbf{A1})$ combined with~$(\mathbf{A2})$ imply the existence of a parameter-to-state mapping for the inner-problem (see section~\ref{sec-cts}). Assumption~$(\mathbf{A3})$ is natural, as we aim at designing regularizers for (inner-)problems that are solvable. From~$(\mathbf{A3})$ we deduce easily existence of minimizers when the regularizer is non-trivial (Proposition~\ref{prop-exist-min-inner}). While addressing the inner-problems only requires first-order differentiability (section~\ref{sec-inner}), the derivation of optimality conditions for the outer-problem (section~\ref{sec-relax0}) necessitates second-order derivatives, in particular for the objective function of the inner-problem, including the cost function $c$ (Assumption~$(\mathbf{A4})$) and the neural network (Assumption~$(\mathbf{A5})$). Throughout the paper these assumptions will be implicitly assumed.

\subsection{On the parameter-to-state mapping for the inner-problem} \label{sec-cts}
The state constraint is given by the implicit function $\varphi: \mathcal{Y} \times \mathcal{U} \rightarrow \mathcal{Q}'$. Under Assumptions~$(\mathbf{A1})$-$(\mathbf{A2})$, we are able to define the parameter-to-state mapping that we denote by $\mathbb{S}$:

\begin{proposition}\label{prop-cts}
Assume $(\mathbf{A1})$-$(\mathbf{A2})$. There exists a mapping $\mathbb{S}: \mathcal{U} \rightarrow \mathcal{Y}$ of class $C^2$ such that for all $(y,u)\in \mathcal{Y} \times \mathcal{U}$ we have 
\begin{eqnarray*}
\varphi(y,u)=0 & \Leftrightarrow & y = \mathbb{S}(u).
\end{eqnarray*}
Furthermore, $\mathbb{S}$ is of class $C^2$ over $\mathcal{U}$.
\end{proposition}

\begin{proof}
The result is due to the implicit function theorem.
\end{proof}

Depending on the context, for the sake of convenience we may consider the original equality constraint $\varphi(y,u) = 0$, or substitute the variable $y$ by $\mathbb{S}(u)$. The lack of parameter-to-state mapping $\mathbb{S}$ for the inner-problem would introduce further difficulties that we choose to not address in this paper.

\section{Necessary optimality conditions for the inner-problem} \label{sec-inner}

Given $w\in \mathcal{W}$, when the regularizer is chosen in the form of a neural network as described in section~\ref{sec-NN0}, the inner-problems correspond to the following class of optimal control problems:
\begin{equation} \label{innerpb} \tag{$\ast$}
\left\{ \begin{array} {l}
\displaystyle \min_{(y,u)\in \mathcal{Y} \times \mathcal{U}} \Big(J(y,u,w) := c(y,\hat{y}) + \gamma \circ r(w,u)\Big), \\[10pt]
\text{subject to: } \varphi(y,u) = 0 \text{ in } \mathcal{Q}'.
\end{array} \right.
\end{equation}
Again, we keep the same notation for this new formulation of the inner-problems~\eqref{innerpb}. The aim of this section is to provide them necessary optimality conditions. For that purpose, we first need to investigate the differential properties owned by neural networks of $\mathcal{W}$.




\subsection{Differential properties of neural networks}
\label{sec-NN}
Consider the partial neural network with $\ell$ layers ($1\leq \ell \leq L$):
\begin{eqnarray*}
r^{(\ell)}(w^{(\ell)},u) & = & 
\left\{ \begin{array} {ll}
w_1(u) & \text{if } \ell = 1, \\
w_2(\rho (w_{1}(u))) 
& \text{if } \ell=2 , \\[5pt]
w_{\ell}(\rho (w_{\ell-1}(\rho(\dots \rho(w_1(u)))))) & \text{if } \ell \geq 3.
\end{array} \right.
\end{eqnarray*}
Recall that the weights $(w_\ell)$ are affine functions of~$\R^{n_\ell}$ with values in~$\R^{n_{\ell+1}}$, of the form $w_\ell (\mathrm{z}) = A_\ell\mathrm{z} + b_\ell$ with $A_\ell \in \R^{n_{\ell+1}\times n_\ell}$ and $b_\ell \in \R^{n_{\ell+1}}$. Only for the present subsection and also subsection~\ref{sec-awful}, we denote
\begin{equation*}
w^{(\ell)} = (w_1, \dots, w_{\ell}) \in 
 \prod_{k =1}^{\ell} \mathrm{Aff}(\R^{n_k},\R^{n_{k+1}}).
\end{equation*}	
Observe that for $1 \leq \ell \leq L-1$ the following formula holds:
\begin{eqnarray}
r^{(\ell+1)}(w^{(\ell+1)},u) & = & w_{\ell+1}\rho\big(r^{(\ell)}(w^{(\ell)},u)\big). \label{formulaPhi}
\end{eqnarray}
From~\eqref{formulaPhi}, by induction we can verify that the sensitivity of $r^{(L)}$ with respect to $u$ satisfies the identity
\begin{equation}
\frac{\p r^{(L)}}{\p u}(w^{(L)},u)  =  A_L
\mathrlap{\prod_{\ell = 1}^{L-1}}{\hspace*{-0.2pt}\longrightarrow} 
 \rho'(r^{(\ell)}(w^{(\ell)},u)) A_{\ell} . \label{diffNNu}
\end{equation}
The vector/matrix product $\rho'(r^{(\ell)}(w^{(\ell)},u)) A_\ell$ is a matrix, and has to be understood as follows:
\begin{equation}
(\rho' \, A)_{ij}  =  \rho'_i \, A_{ij}. \label{trickvec}
\end{equation}
The symbol $\displaystyle \mathrlap{\prod_{\ell = 1}^{L-1}}{\hspace*{-0.2pt}\longrightarrow}$ is used for calculating the non-commutative product of matrices, by multiplying them iteratively to the left when the index $\ell$ increases: $\displaystyle \mathrlap{\prod_{\ell = 1}^{L}}{\hspace*{-1pt}\longrightarrow} B_{\ell} = B_L B_{L-1}\dots B_2 B_1$.
Without ambiguity, we will simply denote $r = r^{(L)}$ and $w = w^{(L)}$. Using the mean-value theorem, by induction we can estimate
\begin{eqnarray*}
\| r(w,u) - r(w,0) \|_{\mathcal{V}} & \leq & 
\left(\prod_{\ell = 1}^L |A_\ell |_{\R^{n_{\ell+1}\times n_\ell}} \right)
\|\rho'\|_{\infty}^{\mathcal{L}_{\mathcal{W}}} \| u \|_{\mathcal{U}}, \\
\|r(w,u_1) - r(w,u_2)\|_{\mathcal{V}} & \leq &
\left(\prod_{\ell = 1}^L |A_\ell |_{\R^{n_{\ell+1}\times n_\ell}} \right)
\|\rho'\|_{\infty}^{\mathcal{L}_{\mathcal{W}}} \| u_1-u_2 \|_{\mathcal{U}},
\end{eqnarray*}
where we recall that we have introduced $\mathcal{L}_{\mathcal{W}} = \displaystyle \sum_{\ell=1}^{L-1} n_{\ell}$ in section~\ref{sec-NN0}. The differentiation of $r$ with respect to the weights $w_\ell$ is given by the following formula, for all $\tilde{w} \in \mathrm{Aff}(\R^{n_\ell},\R^{n_{\ell+1}})$:
\begin{equation*}
\frac{\p r}{\p w_{\ell}}(w,u).\tilde{w} = 
A_L \left(\mathrlap{\prod_{k=\ell}^{L-1}}{\hspace*{1pt}\longrightarrow} \rho'(r^{(k)}(w^{(k)},u))A_k\right) \left( \tilde{w}( \rho(r^{(\ell-1)}(w^{(\ell-1)},u)))\right).
\label{diffNNw}
\end{equation*}
We can obtain this formula also by induction, using~\eqref{formulaPhi}. Here the product has to be understood as follows for $1\leq i \leq n_{\mathrm{out}}$:
\begin{eqnarray*}
\left(\frac{\p r}{\p w_\ell}(w,u).\tilde{w}\right)_i & = & 
\left(A_L\right)_{ij} \left(\mathrlap{\prod_{k=\ell}^{L-1}}{\hspace*{0pt}\longrightarrow} \left(\rho'(r^{(k)}(w^{(k)},u))A_k\right)\right)_{jh} \left( \tilde{w}( \rho\left(r^{(\ell-1)}(w^{(\ell-1)},u)\right))\right)_h.
\end{eqnarray*}
These expressions for the derivatives of the neural network will be used in practice for the implementation of the gradient associated with the inner-problem. Let us now derive expressions for this gradient.

\subsection{Direct approach} \label{sec-direct-app}
The parameter-to-state mapping $\mathbb{S}$ obtained in Proposition~\ref{prop-cts} enables us to derive classically necessary optimality conditions for the inner-problems directly by omitting the state constraint. Indeed, using $y = \mathbb{S}(u)$, and given $w\in \mathcal{W}$, we can rewrite Problem~\eqref{pbstandard} as
\begin{equation*}
\min_{u\in \mathcal{U}} \Big(\tilde{J}(u,w) := J(\mathbb{S}(u),u,w) = c(\mathbb{S}(u),\hat{y}) + \gamma(r(w,u))\Big).
\end{equation*}
Using Assumption~$(\mathbf{A3})$ and the formulation above, we obtain existence of minimizers for the inner-problem~\eqref{innerpb}:

\begin{proposition}\label{prop-exist-min-inner}
Under Assumption~$(\mathbf{A3})$, for all $w\in \mathcal{W}$ Problem~\eqref{innerpb} admits a global minimizer.
\end{proposition}

\begin{proof}
From Assumption~$(\mathbf{A3})$ and Proposition~\ref{prop-cts}, the mapping $u \mapsto c(\mathbb{S}(u), \hat{y})$ admits global minimizers, and due to the regularity of $u \mapsto \gamma(r(w,u))$, the result follows straightforwardly.
\end{proof}

We next obtain necessary optimality conditions involving the derivative of $\mathbb{S}$:

\begin{proposition} \label{prop-optcond-inner}
Let be $w\in \mathcal{W}$. 
If $(y=\mathbb{S}(u),u) \in \mathcal{Y} \times \mathcal{U}$ is solution of Problem~\eqref{pbstandard}, then the gradient
\begin{equation}
G(u,w,\hat{y}) := \frac{\p \tilde{J}}{\p u}(u,w) = \mathbb{S}'(u)^{\ast}c_y(\mathbb{S}(u),\hat{y}) + r'_u(w,u)^{\ast}\gamma'(r(w,u))
\label{exp-grad-direct}
\end{equation}
vanishes.
\end{proposition}

\begin{proof}
Following the previous comments and Proposition~\ref{prop-cts}, the result follows from the Karush-Kuhn-Tucker conditions (see~\cite[section~2.8, p.~63]{Troeltzsch}).
\end{proof}

In certain problems, like shape optimization for example, evaluating efficiently values of the parameter-to-state mapping $\mathbb{S}(u)$, and in particular $\mathbb{S}'(u)^{\ast}$, can be challenging numerically, or expensive computationally. Therefore we prefer sometimes to adopt an optimal control approach, also called the {\it adjoint approach}, that we develop in our context in section~\ref{sec-opt-inner} below. See remarks in~\cite[section~1.6.1, page~59]{Pinnau} for further explanations. Further comments on the possible advantages of the optimal control approach are given in Appendix~\ref{appendix1}.

\subsection{Optimal control formalism for the inner-problem~\eqref{pbstandard}} \label{sec-opt-inner}
Let us recall the initial description of the inner-problem, when we keep the original state constraint:

\begin{equation} \label{innerpb2} \tag{$\ast$}
\left\{ \begin{array} {l}
\displaystyle \min_{(y,u)\in \mathcal{Y} \times \mathcal{U}} \Big(J(y,u,w) := c(y,\hat{y}) + \gamma \circ r(w,u)\Big), \\[10pt]
\text{subject to: } \varphi(y,u) = 0 \text{ in } \mathcal{Q}'.
\end{array} \right.
\end{equation}

Unlike the direct approach which utilizes $\varphi(y,u) = 0 \Leftrightarrow y = \mathbb{S}(u)$, the optimal control approach -- also called the adjoint approach -- consists in taking into account the state constraint by introducing a Lagrange multiplier. We define the Lagrangian of problem~\eqref{innerpb2} as follows
\begin{eqnarray*}
 \begin{array} {rrcl}
\mathscr{L}: & (\mathcal{Y}\times \mathcal{U}) \times \mathcal{Q} & \rightarrow & \R \\
& ((y,u),p) & \mapsto & c(y,\hat{y}) + \displaystyle  \gamma \circ r(w,u) 
+ \left\langle p \, ,  \varphi(y,u) \right\rangle_{\mathcal{Q},\mathcal{Q}'}.
\end{array}
\end{eqnarray*}
Problem~\eqref{innerpb2} can be solved by finding a saddle-point to the Lagrangian $\mathscr{L}$, minimized with respect to variables $(y,u)$ and maximizing with respect to $p$. The sensitivity of~$\mathscr{L}$ with respect to variable $y$ yields the adjoint equation, whose unknown is the adjoint state $p\in \mathcal{Q}$. It is given in $\mathcal{Y}'$ by
\begin{eqnarray}
\varphi'_y(y,u)^\ast.p + c'_y(y,\hat{y}) = 0.
\label{sysadj}
\end{eqnarray}
We define a solution of system~\eqref{sysadj} by transposition.
\begin{definition} \label{defsoladj}
We say that $p$ is a solution of system~\eqref{sysadj} if for all $f\in \mathcal{Q}'$ we have
\begin{eqnarray}
\langle p \, ,f \rangle_{\mathcal{Q},\mathcal{Q}'} + c'_y(y,\hat{y}).\tilde{y} = 0,
\label{eqsoltrans}
\end{eqnarray}
where $\tilde{y}\in \mathcal{Y}$ denotes the solution of the following linear system in $\mathcal{Q}'$:
\begin{eqnarray}
\varphi'_y(y,u).\tilde{y} = f. \label{syslin}
\end{eqnarray}
\end{definition}
From Assumption~$(\mathbf{A2})$, system~\eqref{syslin} is well-posed. The existence of solutions for adjoint system~\eqref{sysadj}, in the sense of Definition~\ref{defsoladj}, is stated as follows:
\begin{proposition}
Let be $(y,u) \in \mathcal{Y} \times \mathcal{U}$. There exists a unique solution $p \in \mathcal{Q}$ to system~\eqref{sysadj}, in the sense of Definition~\ref{defsoladj}, and there exists a constant $C(y,u) >0$ (depending only on $(y,u)$) such that
\begin{eqnarray*}
\| p \|_{\mathcal{Q}} & \leq & C(y,u) \| c'_y(y,\hat{y}) \|_{\mathcal{Y}'}.
\end{eqnarray*}
\end{proposition}

\begin{proof}
From Assumption~$(\mathbf{A2})'$, it is easy to verify that the mapping $\varphi_y(y,u)^\ast$ is injective, and so uniqueness holds. Let us prove existence. Denote by $\Lambda(y,u) : \mathcal{Q}' \rightarrow \mathcal{Y}$ the linear operator which maps $f$ to $\tilde{y}$, where $\tilde{y}$ is the solution of~\eqref{syslin}. From Assumption~$(\mathbf{A2})$, the operator $\Lambda(y,u)$ is well-defined and bounded, and consequently $\Lambda(y,u)^\ast : \mathcal{Y}' \rightarrow \mathcal{Q}'' \simeq \mathcal{Q}$ too. Now define $p = -\Lambda(y,u)^\ast c'_y(y,\hat{y})$. Then, for all $f\in \mathcal{Q}'$, we verify that
\begin{eqnarray*}
\langle p \, , f \rangle_{\mathcal{Q}',\mathcal{Q}} & = & - \langle c'_y(y,\hat{y}) \, , \Lambda(y,u)f \rangle_{\mathcal{Y}',\mathcal{Y}} \\
	& = & -\langle c'_y(y,\hat{y}) \, , \tilde{y} \rangle_{\mathcal{Y}',\mathcal{Y}} = -c'_y(y,\hat{y}).\tilde{y},
\end{eqnarray*}
and the announced estimate follows from boundedness of $\Lambda(y,u)^\ast$, which completes the proof.
\end{proof}

Consequently, the adjoint approach provides another expression for the gradient associated with the inner-problem, namely the quantity $G(u,w,\hat{y})$ introduced in~\eqref{exp-grad-direct}.

\begin{proposition} \label{prop-optcond-inner-adj}
Let be $w\in \mathcal{W}$. If $u$ is solution to Problem~\eqref{pbstandard}, then the gradient $G(u,w,\hat{y})$ introduced in Proposition~\ref{prop-optcond-inner} writes
\begin{equation}
G(u,w,\hat{y}) = 
r'_{u}(w,u)^{\ast} \gamma'(r(w,u))  + \varphi'_u(y,u)^\ast.p,
\label{defgradientstd}
\end{equation}
where $y=\mathbb{S}(u)$, $p$ is the solution of~\eqref{sysadj} corresponding to $(y,u)$, and necessarily $G(u,w) = 0$. 
\end{proposition}

\begin{proof}
From Proposition~\ref{prop-cts}, Problem~\eqref{pbstandard} reduces to $
\displaystyle \min_{u\in \mathcal{U}} J(\mathbb{S}(u),u,w)$. Following the Karush-Kuhn-Tucker conditions, if $u$ is optimal to Problem~\eqref{pbstandard}, then $ \frac{\p}{\p u}\left(J(\mathbb{S}(u),u,w) \right) = 0$,
with for all $\tilde{u} \in U$
\begin{eqnarray}
\frac{\p}{\p u}\left(J(\mathbb{S}(u),u,w) \right).\tilde{u} & = & 
J'_y(\mathbb{S}(u),u,w).(\mathbb{S}'(u).\tilde{u}) +
J'_u(\mathbb{S}(u),u,w).\tilde{u} \nonumber \\
& = & \langle c'_y(y,\hat{y}) \, , \mathbb{S}'(u).\tilde{u} \rangle_{\mathcal{U}}
+ \langle \gamma' (r(w,u)) \, , r'_{u}(w,u).\tilde{u} \rangle_{\mathcal{V}} \nonumber \\
& = & \langle \mathbb{S}'(u)^{\ast}c'_y(y,\hat{y}) + r'_{u}(w,u)^\ast\gamma'(r(w,u))^{\ast} \, , \tilde{u} \rangle_{\mathcal{U}} \nonumber \\
& = &  \langle -\mathbb{S}'(u)^{\ast}\varphi'_y(y,u)^{\ast}.p + r'_{u}(w,u)^\ast \gamma'(r(w,u)) \, , \tilde{u} \rangle_{\mathcal{U}},\label{eqproof}
\end{eqnarray}
where $y = \mathbb{S}(u)$ and $p$ satisfies~\eqref{sysadj}. By differentiating the identity $\varphi(\mathbb{S}(u),u) = 0$, for all $\tilde{u} \in \mathcal{U}$ we obtain
\begin{equation}
\varphi'_u(y,u).\tilde{u} + \varphi_y'(y,u).(\mathbb{S}'(u).\tilde{u})  =  0, 
\label{eq-state-diff}
\end{equation}
which implies that
\begin{equation*}
-\mathbb{S}'(u)^{\ast}\varphi'_y(y,u)^{\ast} = \varphi'_u(y,u)^{\ast},
\end{equation*}
and thus~\eqref{eqproof} reduces to 
\begin{equation*}
\frac{\p}{\p u}\left(J(\mathbb{S}(u),u,w) \right).\tilde{u}
 =  \langle \varphi'_u(y,u)^{\ast}.p + r'_u(w,u)^\ast \gamma'(r(w,u))
 \, , \tilde{u}\rangle_{\mathcal{U}',\mathcal{U}} = 0
\end{equation*}
for all $\tilde{u} \in \mathcal{U}$, which completes the proof.
\end{proof}




\begin{remark} \label{remark1}
One could also have obtained the necessary conditions of Proposition~\ref{prop-optcond-inner-adj} by considering a stationary point of the Lagrangian functional $\mathscr{L}$. Indeed, vanishing the derivative of $\mathscr{L}$ with respect to $p$ yields $\varphi(y,u) = 0  \Leftrightarrow  y = \mathbb{S}(u)$,
in virtue of Proposition~\ref{prop-cts}. Vanishing the derivative of $\mathscr{L}$ with respect to $y$ leads to the adjoint system~\eqref{sysadj}. With the chain rule, we calculate for all $\tilde{u} \in  \mathcal{U}$:
\begin{eqnarray}
\frac{\p \mathscr{L}}{\p u}(w,u).\tilde{u} & = & \langle \gamma'(r(w,u)) \, ,  r_u'(w,u).\tilde{u} \rangle_{\mathcal{V}}
 + c'_y(\mathbb{S}(u),\hat{y}).(\mathbb{S}'(u).\tilde{u})  \nonumber\\
 & = & \langle r'_u(w,u)^\ast \gamma'(r(w,u)) \, , \tilde{u} \rangle_{\mathcal{U}}
 - \langle \varphi'_y(\mathbb{S}(u),u)^\ast p \, , \mathbb{S}'(u).\tilde{u}\rangle_{\mathcal{Y}',\mathcal{Y}} \nonumber \\
 & = & \langle r'_u(w,u)^\ast \gamma'(r(w,u)) \, , \tilde{u} \rangle_{\mathcal{U}}
 - \langle  p \, , \varphi'_y(\mathbb{S}(u),u).(\mathbb{S}'(u).\tilde{u})\rangle_{\mathcal{U}}, \label{eqproof2}
\end{eqnarray}
where $p$ is solution of~\eqref{sysadj} with $y = \mathbb{S}(u)$. Using~\eqref{eq-state-diff}, we see that~\eqref{eqproof2} reduces to 
\begin{equation*}
\frac{\p \mathscr{L}}{\p u}(w,u)  =  \langle r'_u(w,u)^\ast \gamma'(r(w,u))\, , \tilde{u}\rangle_{\mathcal{U}} 
+ \langle p \, , \varphi'_u(y,u). \tilde{u}\rangle_{\mathcal{U}',\mathcal{U}}
= \langle G(u,w,\hat{y}),\tilde{u} \rangle_{\mathcal{U}',\mathcal{U}},
\end{equation*}
which leads to the same result as Proposition~\ref{prop-optcond-inner-adj}. Therefore the identities given in Proposition~\ref{prop-optcond-inner-adj} can be summarized as
\begin{equation*}
\left(\frac{\p \mathscr{L}}{\p y}(y,u,p),\frac{\p \mathscr{L}}{\p u}(y,u,p),\frac{\p \mathscr{L}}{\p p}(y,u,p) \right) = 0.
\label{stdnecoptsum}
\end{equation*}
\end{remark}


The necessary optimality conditions given in Proposition~\ref{prop-optcond-inner} are self-consistent, and thus more appropriate for theoretical analysis, while those provided by Proposition~\ref{prop-optcond-inner-adj} are more convenient more numerical realization. Following these two results, the constraint of Problem~\eqref{pbmain} is then {\it relaxed} in the next section, by being simply replaced by the identity $G(u,w,\hat{y}) = 0$.

\section{Relaxation approach for the outer-problem
} \label{sec-relax0}


\subsection{Relaxed formulation of the main problem} \label{sec-relaxed-outer}
The aim of the present section is to translate the abstract constraint of Problem~\eqref{pbmain}, namely "{\it $u_k$ is solution of Problem~\eqref{pbstandard}}[\dots]", into an equality constraint. For that purpose, we use the necessary vanishing gradient condition given in Propositions~\ref{prop-optcond-inner} and~\ref{prop-optcond-inner-adj}. Note that a solution to Problem~\eqref{pbstandard} with $(\hat{y},r) = (\hat{y}_k,r(w,\cdot))$ satisfies necessarily $G(u_k,w,\hat{y}_k) = 0$, but if it satisfies this identity only, then in general this is not necessarily a solution of Problem~\eqref{pbstandard}. Therefore we consider the relaxed version of~\eqref{pbmain} below:

\begin{equation} \label{pbmainrelaxed} \tag{$\tilde{\ast\ast}$}
\left\{ \begin{array} {l}
\displaystyle \min_{w\in \mathcal{W}} \quad
\frac{1}{2K} \sum_{k=1}^K \| u_k - \hat{u}_k \|_{\mathcal{U}}^2
+\frac{\nu}{2} \|w\|^2_{\mathcal{W}} 
, \\[15pt]
\text{subject to: $G(u_k,w,\hat{y}_k) = 0$ for all $1\leq k\leq K$.}
\end{array} \right.
\end{equation}
The parameter $\nu \geq 0$ is given, and introduced for theoretical purpose mainly (see Theorem~\ref{th-exist-outer}). This regularization term for the set of weights $w$ could also facilitate the numerical solving of Problem~\eqref{pbmainrelaxed}, by forcing the weights to be bounded. Note that a parameter-to-solution mapping $w \mapsto u$ characterizing $G(u,w,\hat{y}) = 0$ is not always available. Even if such a mapping exists, its derivative is difficult to describe explicitly. See section~\ref{sec-regression} for more details. For taking into account the equality constraints of problem~\eqref{pbmainrelaxed}, we adopt an adjoint approach and introduce the multipliers $\mu \in \mathcal{U}$, with the following Lagrangian mapping:
\begin{eqnarray}
\begin{array} {rccl}
\mathbf{L}: & \mathcal{U} \times \mathcal{W} \times \mathcal{U} & \rightarrow & \R \\[5pt]
 & (u,w, \mu) & \mapsto &  \displaystyle\frac{1}{2}\| u - \hat{u} \|_{\mathcal{U}}^2
+ \frac{\nu}{2} \| w\|^2_{\mathcal{W}}
 + \left\langle \mu \, , G(u,w,\hat{y}) \right\rangle_{\mathcal{U},\mathcal{U}'}
\end{array}
\label{myLag}
\end{eqnarray}
Like in section~\ref{sec-opt-inner} for the inner-problems (see Remark~\ref{remark1}), a solution of Problem~\eqref{pbmainrelaxed} can be sought -- when $K=1$ -- as a saddle-point of $\mathbf{L}$ with respect to the variables $(u,w)$ and $\mu$. In practice, we shall consider $(u_k)_{1\leq k \leq K}$, $\mu = (\mu_k)_{1\leq k\leq K}$ and
\begin{eqnarray}
\begin{array} {rccl}
\mathbf{L}^{(K)}: & \mathcal{U}^K \times \mathcal{W} \times \mathcal{U}^K & \rightarrow & \R \\[5pt]
 & (u,w, \mu) & \mapsto &  \displaystyle\frac{1}{2K}
\sum_{k=1}^K\| u_k - \hat{u}_k  \|_{\mathcal{U}}^2
+ \frac{\nu}{2} \| w\|^2_{\mathcal{W}}
 + \frac{1}{K}\sum_{k=1}^K\left\langle \mu_k \, , G(u_k,w,\hat{y}_k) \right\rangle_{\mathcal{U},\mathcal{U}'}.
\end{array}
\label{superLK}
\end{eqnarray}
Up to a change of the functional framework, for the theoretical questions we can simply consider the functional introduced in~\eqref{myLag}, by considering the vector version of the different variables, namely $u = (u_k)_{1\leq k\leq K}$, $\hat{u} = (\hat{u}_k)_{1\leq k\leq K}$ and $\mu = (\mu_k)_{1\leq k\leq K}$. 
Before deriving conditions of stationarity for $\mathbf{L}$, let us discuss about existence of solutions for Problem~\eqref{pbmainrelaxed}.

\subsection{Existence of minimizers for the relaxed problem} \label{sec-existmin} \label{sec-exist-outer}

Existence of solutions for Problem~\eqref{pbmainrelaxed} follows from standard arguments:

\begin{theorem} \label{th-exist-outer}
Assume that $\nu >0$. Then Problem~\eqref{pbmainrelaxed} admits a global minimizer.
\end{theorem}

\begin{proof}
The existence of feasible solutions for Problem~\eqref{pbmainrelaxed} is due to Proposition~\ref{prop-exist-min-inner}, combined with Proposition~\ref{prop-optcond-inner} (or also Proposition~\ref{prop-optcond-inner-adj}). Since the functional
\begin{equation}
\mathcal{U} \times \mathcal{W} \ni (u,w) \mapsto \frac{1}{2}\|u-\hat{u}\|^2_{\mathcal{U}} + \frac{\nu}{2}\|w\|_{\mathcal{W}}^2
\label{nice-func}
\end{equation}
is bounded from below for~$\nu>0$, Problem~\eqref{pbmainrelaxed} admits a minimizing sequence $(\mathrm{u}^{(n)}, w^{(n)})_{n}$ of feasible solutions. Further, from~\eqref{nice-func}, the sequences~$(u^{(n)})_n$ and~$(w^{(n)})_n$ are bounded in $\mathcal{U}$ and $\mathcal{W}$, respectively. Remind that the spaces $\mathcal{U}$ and $\mathcal{W}$ are finite-dimensional. Consequently, up to extraction, there exists $(u,w) \in \mathcal{U} \times \mathcal{W}$ such that $(u^{(n)},w^{(n)})_n$ converges strongly towards $(u,w)$. It is clear that $\mathcal{U} \times \mathcal{W} \ni (u,w) \mapsto G(u,w,\hat{y}) \in \mathcal{U}'$ is continuous, so that we can pass to the limit in $G(u^{(n)},w^{(n)},\hat{y}) = 0$ for obtaining $G(u,w,\hat{y}) = 0$, which concludes the proof, as $(u,w)$ is then solution of Problem~\eqref{pbmainrelaxed}.
\end{proof}

\begin{remark}
If we consider the case $\nu = 0$, the boundedness of the sequence~$(w_n)_n$ in the proof above is {\it a priori} not guaranteed. For proceeding like we did, in such a case we may need to assume further hypotheses on the neural network. For example, we could exploit the equality $G(u^{(n)},w^{(n)},\hat{y}) = 0$, that is
\begin{equation*}
\mathbb{S}'(u^{(n)})^{\ast}c_y(\mathbb{S}(u^{(n)}),\hat{y}) + r'_u(w^{(n)},u^{(n)})^{\ast}\gamma'(r(w^{(n)},u^{(n)})) = 0,
\end{equation*}
due to the identity~\eqref{exp-grad-direct}. The sequence~$(u^{(n)})_n$ still converges, which implies that the quantity $$r'_u(w^{(n)},u^{(n)})^{\ast}\gamma'(r(w^{(n)},u^{(n)}))$$ 
is uniformly bounded with respect to $n$. In the simple one-dimensional case where $n_{\mathrm{in}} = n_{\mathrm{out}} = 1$, and $\gamma \equiv \mathrm{Id}$, this would lead us to consider that the boundedness of the gradient of the neural network implies the boundedness of its weights, which is not true in general. Therefore considering $\nu >0$ seems to be reasonable, as this provides a simple proof to the existence of minimizers. 
\end{remark}

\subsection{Derivatives of second-order for the neural network} \label{sec-awful}
In section~\ref{sec-opt-outer}, we will need expressions of the second-order derivatives of the neural network. From the formula~\eqref{diffNNu} obtained in section~\ref{sec-NN}, we derive
\begin{equation}
r''_{uu}(w,u) = A_L \sum_{\ell=1}^{L-1} \left(\left( \mathrlap{\prod_{k=1}^{\ell-1}}{\hspace*{-0.2pt}\longrightarrow} 
\rho'(r^{(k)}) A_k
\right)
\rho''(r^{(\ell)}) \frac{\p r^{(\ell)}}{\p u} A_{\ell}
\left( 
\mathrlap{\prod_{k=\ell+1}^{L-1}}{\hspace*{4.5pt}\longrightarrow} 
\rho'(r^{(k)})A_k
\right)\right), \label{diffNNuu}
\end{equation}
and
\begin{equation}
r''_{u{w_s}}(w,u).\tilde{w} = A_L \sum_{\ell=1}^{L-1} \left(\left( \mathrlap{\prod_{k=1}^{\ell-1}}{\hspace*{-0.2pt}\longrightarrow} 
\rho'(r^{(k)}) A_k
\right)
\frac{\p}{\p w_s}\left(\rho'(r^{(\ell)}(w^{(\ell)},u)) A_{\ell} \right).\tilde{w}
\left( 
\mathrlap{\prod_{k=\ell+1}^{L-1}}{\hspace*{4.5pt}\longrightarrow} 
\rho'(r^{(k)})A_k
\right)
\right) \label{diffNNuw}
\end{equation}
for all $1\leq s \leq L-1$, where we have simply denoted $r^{(k)} = r^{(k)}(w^{(k)},u)$. We recall
\begin{eqnarray*}
\frac{\p r^{(\ell)}}{\p u}(w^{(\ell)},u)  = A_L\mathrlap{\prod_{k = 1}^{\ell-1}}{\hspace*{0pt}\longrightarrow} \rho'(r^{(k)}) A_{k}, & & 
\frac{\p r^{(\ell)}}{\p w_{s}}(w^{(\ell)},u).\tilde{w}  = 
A_L \left(\mathrlap{\prod_{k=s}^{\ell-1}}{\hspace*{0pt}\longrightarrow} \rho'(r^{(k)})A_k\right) \left( \tilde{w}( \rho(r^{(s-1)}))\right),
\end{eqnarray*}
and for all $\tilde{w} :\mathrm{z} \mapsto \tilde{A}\mathrm{z} + \tilde{b}$
\begin{equation*}
\frac{\p}{\p w_s}\left(\rho'(r^{(\ell)}(w^{(\ell)},u)) A_{\ell} \right).\tilde{w} = 
\displaystyle 
\left\{
\begin{array} {cl}
\displaystyle 0 & \text{if } s>\ell, \\[10pt]
\displaystyle 
\rho'(r^{(\ell)}(w^{(\ell)},u))\tilde{A} + 
\rho''(r^{(\ell)})
\left(\tilde{w}\rho(r^{(\ell-1)})\right) A_{\ell}
& \text{if } s = \ell, \\[10pt]
\displaystyle \rho''(r^{(\ell)})
\left(\frac{\p r^{(\ell)}}{\p w_s}.\tilde{w}\right) A_{\ell} & \text{if } s < \ell.
\end{array}
\right.
\end{equation*}
Like for~\eqref{trickvec}, the products of type $(\rho''\tilde{r})_i = \rho''_i\tilde{r}_i$ are defined componentwise. For $s = L$, we have for all $\tilde{w} :\mathrm{z} \mapsto \tilde{A}\mathrm{z} + \tilde{b}$
\begin{eqnarray*}
r''_{u{w_L}}(w,u).\tilde{w} = \tilde{A}\mathrlap{\prod_{k = 1}^{\ell-1}}{\hspace*{0pt}\longrightarrow} \rho'(r^{(k)}) A_{k}.
\end{eqnarray*}
These expressions in themselves are not essential for what follows in the present section, but they can help the reader to reproduce the numerical results we obtained in section~\ref{sec-num}.

\subsection{Properties of the adjoint operator} \label{sec-regression}

The adjoint equation for Problem~\eqref{pbmainrelaxed} is obtained by differentiating the Lagrangian~$\mathbf{L}$ with respect to the variable~$u$, which gives
\begin{equation}
G'_u(u,w,\hat{y})^{\ast}. \mu = -(u-\hat{u}).
\label{eq-superadj0}
\end{equation}
Let us take a look at the operator~$G'_u(u,w,\hat{y})^{\ast} \in \mathcal{L}(\mathcal{U},\mathcal{U}')$ which arises in this equation. Since $\mathcal{U} \simeq \R^{n_{\mathrm{in}}}$ is assumed to be of finite dimension, for each $(u,w) \in \mathcal{U} \times \mathcal{W}$ the operator $G'_u(u,w,\hat{y})$ can be represented by a matrix of $\R^{n_{\mathrm{in}} \times n_{\mathrm{in}}}$. Further, reminding that $G(u,w,\hat{y}) = \frac{\p }{\p u} J(\mathbb{S}(u),u,w)$, it yields that $G'_u(u,w,\hat{y})$ is an Hessian matrix, and so it is represented by a symmetric matrix. Let us give the expression of~$G'_u(u,w,\hat{y})$, by denoting $y = \mathbb{S}(u)$ :
\begin{equation*}
G'_u(u,w,\hat{y})  =  \mathbb{S}'(u)^{\ast}c''_{yy}(y,\hat{y})\mathbb{S}'(u) +
c'_y(y,\hat{y})\mathbb{S}''(u) + \gamma'(r(w,u))r_{uu}''(w,u) + r'(w,u)^{\ast} \gamma''(r(w,u)) r'(w,u).
\end{equation*}
More specifically, for all $v_1$, $v_2 \in \mathcal{U}$ we have:
\begin{eqnarray*}
\left\langle G'_u(u,w,\hat{y}).v_1 ; v_2\right\rangle_{\mathcal{U}',\mathcal{U}} & = & c''_{yy}(y,\hat{y}).(\mathbb{S}'(u).v_1, \mathbb{S}'(u).v_2) +
c'_y(y,\hat{y}).(\mathbb{S}''(u).(v_1,v_2)) \\
& & +  \gamma''(r(w,u)).( r'(w,u).v_1, r'(w,u).v_2)
+\gamma'(r(w,u)).(r_{uu}''(w,u).(v_1,v_2)).
\end{eqnarray*}
Solving equation~\eqref{eq-superadj0} determines the value of the variable~$\mu$, and thus the question arises whether the operator $G'_u(u,w,\hat{y})^{\ast} = G'_u(u,w,\hat{y})$ is invertible.  Let us comment on the general case where $G'_u(u,w,\hat{y})$ is not necessarily invertible.

\subsubsection{When the adjoint operator is invertible}
Consider any $(u,w) \in \mathcal{U} \times  \mathcal{W}$ such that $G'_u(u,w,\hat{y})$ is an invertible matrix. Analogously to the inner-problems, such a case corresponds to an assumption of type~$(\mathbf{A2})$ transcribed to Problem~\eqref{pbmainrelaxed}, and following section~\ref{sec-cts} we can prove via the implicit function theorem the existence of a mapping $w\mapsto u$ characterizing the equality $G(u,w,\hat{y}) = 0$. This is a mapping of type {\it control-to-state}, since for the outer-problem the variable $u$ plays the role of a {\it state} and $w$ is the {\it command}. This case corresponds to variables $w\in \mathcal{W}$ such that the associated inner-problem admits a unique solution $u\in \mathcal{U}$. Since $G'_u(u,w,\hat{y})$ is a square matrix, the matrix $G'_u(u,w,\hat{y})^{\ast}$ is also invertible, and so the adjoint equation~\eqref{eq-superadj0}, namely $G'_u(u,w,\hat{y})^{\ast}\mu + (u-\hat{u})  =0$ in the generic case, admits the unique solution
\begin{equation*}
\mu = -G'_u(u,w,\hat{y})^{-T}(u-\hat{u}).
\end{equation*}
Therefore in this case we are able to derive the necessary Karush-Kuhn-Tucker conditions for Problem~\eqref{pbmainrelaxed}, as the associated qualification constraints are satisfied without ambiguity (see Theorem~\ref{theorem-main}).

\subsubsection{When the adjoint operator is not invertible} \label{sec-trick}

In the case where $G'_u(u,w,\hat{y})$ is not invertible, in practice we introduce a perturbation of the latter, based on the following basic result:

\begin{lemma}
Let $A$ be a square matrix. There exists $\varepsilon >0$ such that $A-\varepsilon \I$ is invertible.
\end{lemma}

\begin{proof}
Denote by $\chi(\lambda) := \det(A-\lambda \I)$ the characteristic polynomial of $A$. If $\mathrm{Sp(A)} = \{0\}$, any $\varepsilon \neq 0$ is suitable. Otherwise, choose any $\varepsilon$ such that $0 < \varepsilon < \inf\big\{ |\lambda |, \ \lambda \in \mathrm{Sp}(A) \setminus \{0\} \big\}$. 
Then we have $\det(A-\varepsilon \I) = \chi(\varepsilon) \neq 0$, which concludes the proof.
\end{proof}

Therefore, when $G'_u(u,w,\hat{y})$ is singular, we can find $\varepsilon(u,w,\hat{y}) > 0$ small enough such that $G'_u(u,w,\hat{y}) - \varepsilon(u,w,\hat{y}) \I$ is invertible. Note that, up to adapting the proof above, we can also choose $\varepsilon(u,w,\hat{y}) > 0$ small enough such that $G'_u(u,w,\hat{y}) + \varepsilon(u,w,\hat{y}) \I$ is invertible. Thus we replace $G'_u(u,w,\hat{y})$ by $G'_u(u,w,\hat{y}) \pm \varepsilon(u,w,\hat{y}) \I$. Note that the data~$\hat{y}_k$ are of finite number, and thus the corresponding solutions~$u_k$ of the inner-problems are of finite number too. Therefore $\varepsilon >0$ can be chosen independent of~$u$ and $\hat{y}$, but not on the number of data~$K$. In practice, we can choose $\varepsilon = \varepsilon(w) >0$ as small as desired, as long as $G'_u(u,w,\hat{y}) \pm \varepsilon(w) \I$ is invertible. This is equivalent to add the small $L^2$-type regularizer $u\mapsto \mp \frac{\varepsilon(w)}{2}\|u\|_{\mathcal{U}}^2$ to the functional of the inner-problem~\eqref{pbstandard}.




\subsection{Necessary optimality conditions for the relaxed problem} \label{sec-relax} \label{sec-opt-outer}


The following result gives necessary conditions that an optimal solution of Problem~\eqref{pbmainrelaxed} has to fulfill.

\begin{theorem} \label{theorem-main}
Assume that $w\in \mathcal{W}$ is a solution of Problem~\eqref{pbmainrelaxed}. Then for all $1\leq k \leq K$ we have $
G(u_k,w,\hat{y}_k) = 0$, and necessarily
\begin{equation}
\mathbf{G}_{\ell}(w) := \nu w_{\ell} + \frac{1}{K}\sum_{k=1}^K G'_{w_{\ell}}(u_k,w,\hat{y}_k)^{\ast}. \mu_k = 0
\label{super-gradient}
\end{equation}
for all $1\leq \ell \leq L$, where $\mu_k$ is the solution of
\begin{equation}
 G'_u(u_k,w,\hat{y}_k)^{\ast}. \mu_k+ (u_k-\hat{u}_k) = 0 \label{eqadjouter-super} 
\end{equation}
for each $1\leq k \leq K$.
\end{theorem}

\begin{proof}
Since the linear mappings $G'_u(u_k,w,\hat{y}_k) \in \mathcal{L}(\mathcal{U},\mathcal{U}')$ are invertible for each $1\leq k \leq K$, or can be perturbed such that they become invertible (see section~\ref{sec-trick}), the qualification constraints of the Karush-Kuhn-Tucker conditions are satisfied for the Lagrangian~\eqref{superLK}. Therefore a solution of Problem~\eqref{pbmainrelaxed} is necessary a critical point of $\mathbf{L}^{(K)}$, which leads to the announced identities.
\end{proof}

In practice, we solve Problem~\eqref{pbmainrelaxed} iteratively, by updating the weights $w$ such that the norm of~\eqref{super-gradient} converges to zero. Given a set of weights~$w$, we first compute each $u_k$ solution of~\eqref{pbstandard} corresponding to $(\hat{y},r) = (\hat{y}_k, r(w,\cdot))$. Note that the variable $u = (u_k)_{1\leq k \leq K}$ plays the role of a state variable in Problem~\eqref{pbmainrelaxed}. Next we compute $\mu = (\mu_k)_{1\leq k\leq K}$, where each $\mu_k$ is solution of~\eqref{eqadjouter-super}. Finally we are able to determine the values of the left-hand-sides~$\mathbf{G}_{\ell}(w)$ in~\eqref{super-gradient}, that we use in a gradient rule.

\newpage

\section{A gradient-based algorithm} \label{sec-algo}

This section is devoted to algorithmic implementation of the relaxed approach for Problem~\eqref{pbmain}. Let us first explain how the inner-problems~\eqref{pbstandard} can be solved in practice. 

\subsection{Nesterov gradient descent for solving the inner-problems}
We solve the inner-problems with Nesterov algorithm~\cite{Nesterov}, also called the accelerated gradient method. It offers a nice compromise between robustness and rapidity of convergence, as inner-problems need to be solved many times, when solving the outer-problem. The corresponding method is given in Algorithm~\ref{algo-inner} below.

\setcounter{algorithm}{-1}
\begin{algorithm}[htpb]
\hfill 
	\begin{description}
		\item[Initialization:] \hfill
		\begin{description}
			\item[-] Initialize $u^{(0)}$ as $\displaystyle \frac{1}{K}\sum_{k=1}^K \hat{u}_k$.
		\end{description}

		\item[Initial gradient:] From $u^{(0)}$, compute the (initial) gradient as follows: \\[5pt]
		\begin{tabular} {lr}
		$\left.\begin{minipage}{11.5cm}
		\begin{description}
			\item[-] Compute the state $y$ corresponding to $u=u^{(0)}$, by solving $\varphi(y,u) = 0$.
			\item[-] Compute the adjoint state $p$ as solution of $\varphi'_y(y,u)^{\ast}.p = - c'_y(y,\hat{y}) $.
			\item[-] Compute the gradient $G^{(0)}= G(u,w,\hat{y})$ using formula~\eqref{defgradientstd}:\\
$\displaystyle G^{(0)}=  r'_{u}(w,u)^{\ast} \gamma'(r(w,u))  + \varphi'_u(y,u)^\ast.p$
		\end{description}
		\end{minipage} 
		\right\}$ & 
		\begin{minipage}{7cm}
		Compute the gradient  \\ in 3 steps
		\end{minipage}
		\end{tabular}
		\hfill \\Store $u^{(0)}$ and $G^{(0)}$.

		\item[Armijo rule:] Choose $\alpha = 0.5$. 
		\begin{description}
			\item[-] Find the smallest $n\in\N$ such that $J(u^{(0)}-\alpha^n G^{(0)}) < J(u^{(0)})$.
			\item[-] Define $u^{(1)} = u^{(0)} - \alpha^n G^{(0)}$. 
			\item[-] Compute the gradient $G^{(1)}$ as above, corresponding to $u^{(1)}$.
			\item[-] Store $u^{(1)}$ and $G^{(1)}$.
		\end{description}

	\item[Nesterov gradient steps:] Initialization with $(u^{(0)},G^{(0)})$ and $(w^{(1)},G^{(1)})$.\\[5pt]
	Compute iteratively $u^{(n)}$ ($n\geq 2$) with the Nesterov steps.\\
	While $|| G(u^{(n)},w,\hat{y})||_{\mathcal{U}} > 1.e^{-10}$, do gradient steps.

	\item[End:]
	Obtain $u$, approximated solution of~\eqref{pbstandard} with $(\hat{y},r) = (\hat{y},r(w,\cdot))$.
	\end{description}
	\caption{Solving the inner-problem for a given regularizer $u\mapsto r(w,u)$ via first-order necessary optimality conditions.}\label{algo-inner}
\end{algorithm}
\FloatBarrier 

Note that for evaluating the gradient~\eqref{super-gradient} of the outer-problem, Algorithm~\ref{algo-inner} needs to be performed $K$ times, since we need to determine each $u_k$ corresponding to $\hat{y}_k$. For a given~$w\in \mathcal{W}$, we denote the evaluation of functional of Problem~\eqref{pbmainrelaxed} as follows
\begin{equation*}
\mathcal{J}(w) := \frac{1}{2K} \sum_{k=1}^K \| u_k - \hat{u}_k \|_{\mathcal{U}}^2
+\frac{\nu}{2} \|w\|^2_{\mathcal{W}},
\end{equation*}
where the $u_k$ are obtained with Algorithm~\ref{algo-inner}.

\newpage
\subsection{Barzilai-Borwein gradient rule for the relaxed outer-problem}
We solve the outer-problem with the Barzilai-Borwein algorithm~\cite{BarBor}. The corresponding method uses Theorem~\ref{theorem-main}, and is explained in Algorithm~\ref{algo-outer} below, in the case $\nu = 0$.
 
\begin{algorithm}[htpb]
\hfill 
	\begin{description}
		\item[Initialization:] \hfill
		\begin{description}
			\item[-] Load/create the data set $\{(\hat{y}_k,\hat{u}_k); 1\leq k \leq K\}$.
			\item[-] Choose random weights as initial weights $w^{(0)}$ for the neural network.
		\end{description}

		\item[Initial gradient:] From $w^{(0)} \in \mathcal{W}$, compute the (initial) gradient. For each $1\leq k \leq K$:\\[5pt]
		\begin{tabular} {lr}
		$\left.\begin{minipage}{11.5cm}
		\begin{description}
			\item[-] With Alogirthm~\ref{algo-inner}, find the approximated solution $u_k$ of the inner-problem~\eqref{pbstandard} corresponding to $(\hat{y},r) = (\hat{y}_k, r(w^{(0)},\cdot))$.
			\item[-] Compute the multipliers $\mu_k$ solutions of~\eqref{eqadjouter-super}:\\ 
$G'_u(u_k,w^{(0)},\hat{y}_k)^{\ast}.\mu_k = -(u_k-\hat{u}_k)$.
			\item[-] Compute the gradient $\mathbf{G}_{\ell}^{(0)}$ using formula~\eqref{super-gradient}:\\
$\mathbf{G}_{\ell}^{(0)}= \displaystyle \frac{1}{K}\sum_{k=1}^K
G'_{w_{\ell}}(u_k,w^{(0)},\hat{y}_k)^{\ast}.\mu_k$
		\end{description}
		\end{minipage} 
		\right\}$ & 
		\begin{minipage}{7cm}
		Compute the gradient  \\ in 3 steps
		\end{minipage}
		\end{tabular}
		\hfill \\Store $w^{(0)}$ and $\mathbf{G}^{(0)} = (\mathbf{G}_{1}^{(0)}, \dots, \mathbf{G}_{L}^{(0)})$.

		\item[Armijo rule:] Choose $\beta = 0.5$. 
		\begin{description}
			\item[-] Find the smallest $n\in\N$ such that $\mathcal{J}(w^{(0)}-\beta^n \mathbf{G}^{(0)}) < \mathcal{J}(w^{(0)})$.
			\item[-] Define $w^{(1)} = w^{(0)} - \beta^n \mathbf{G}^{(0)}$. 
			\item[-] Compute the gradient $\mathbf{G}^{(1)}$ as above, corresponding to $w^{(1)}$.
			\item[-] Store $w^{(1)}$ and $\mathbf{G}^{(1)}$.
		\end{description}

	\item[Barzilai-Borwein steps:] Initialization with $(w^{(0)},\mathbf{G}^{(0)})$ and $(w^{(1)},\mathbf{G}^{(1)})$.\\[5pt]
	Compute iteratively $w^{(n)}$ ($n\geq 2$) with the Barzilai-Borwein steps.\\
	While $\displaystyle || \mathbf{G}(w^{(n)})||_{\mathcal{W}} > 1.e^{-8}$, do gradient steps.

	\item[End:]
	Obtain $w \in \mathcal{W}$, approximated solution of~\eqref{pbmainrelaxed}.
	\end{description}
	\caption{Solving the first-order optimality conditions for the outer-problem.}\label{algo-outer}
\end{algorithm}
\FloatBarrier

The Barzilai-Borwein is not a monotonous method, but presents second-order convergence behavior, as the expressions of its gradient steps can also be obtained via a quasi-Newton method, more precisely a BFGS method (see~\cite[Chapter~6]{Nocedal} for more details).

\newpage
\section{Illustration: Identification of coefficients in an elliptic equation} \label{sec-num}

In this section we illustrate the relaxation approach previously explained by considering the example of coefficient identification in an elliptic partial differential equations.

\subsection{Presentation of the problem} 

We consider the elliptic equation identification problem in a smooth bounded domain $\Omega$ of $\R^d$ ($d\geq 1$). Denote by $\L^{\infty}_F(\Omega)$ any finite-dimensional subspace of $\L^{\infty}(\Omega)$, with $\dim(\L^{\infty}_F(\Omega)) = n_{\mathrm{in}}$. For instance, $\L^{\infty}_F(\Omega)$ can be generated by piecewise constant functions on a given grid of $\Omega$. Define for $M>m>0$ the following set:
\begin{equation}
\mathcal{U}_{m,M} = \left\{u\in \L_F^{\infty}(\Omega)\mid  m\leq u(x) \leq M, \text{ for a.e. }x\in \Omega \right\}. \label{def-UmM}
\end{equation}
The parameter-to-state mapping is given by $\mathbb{S}: \mathcal{U}:= \mathcal{U}_{m,M} \ni u \mapsto y \in \mathcal{Y}:= \H^1(\Omega)$, where $y$ satisfies the following system in $\mathcal{Q}' = \H^{-1}(\Omega) \times \H^{1/2}(\p \Omega)$, with $\mathcal{Q} := \H^1_0(\Omega) \times \H^{-1/2}(\p \Omega)$:
\begin{equation}
\left\{ \begin{array} {rl}
-\divg(u\nabla y) = f & \text{in } \Omega, \\[5pt]
 y_{|\p \Omega} = g &  \text{on } \p \Omega.
\end{array} \right.
\label{sys-a}
\end{equation}
In system~\eqref{sys-a} the right-hand-sides $f$ and $g$ are given. Standard results related to system~\eqref{sys-a} are presented in Appendix~\ref{appendix2}.

\subsubsection{On the inner-problem}
Given $w\in \mathcal{W}$, as inner-problem~\eqref{pbstandard} we consider the following one:
\begin{equation} \label{pbstandard-calderon} \tag{$\ast$}
\left\{ \begin{array} {l}
\displaystyle \min_{(y,u)\in \mathcal{Y} \times \mathcal{U}} \quad
\frac{1}{2} \| y-\hat{y} \|_{\L^2(\Omega)}^2 + r(w,u)
, \\[15pt]
\text{subject to~\eqref{sys-a}.}
\end{array} \right.
\end{equation}
Here~\eqref{sys-a} corresponds formally to $\varphi(y,u) := (
-\divg(u\nabla y) - f, \ y_{|\p \Omega} - g) \in \mathcal{Q}'$. Note that we have chosen $\gamma \equiv \Id$, and thus necessarily $n_{\mathrm{out}} = 1$. Following Proposition~\ref{prop-optcond-inner-adj}, the gradient for problem~\eqref{pbstandard-calderon} is given by
\begin{equation*}
G(u,w,\hat{y}).\tilde{u} = r'_u(w,u).\tilde{u} + \int_{\Omega} \tilde{u} \nabla y \cdot \nabla p \, \d \Omega,
\end{equation*}
and the adjoint state $\left( p, q=\displaystyle u \frac{\p p}{\p n} \right)\in \mathcal{Q} =  \H^1_0(\Omega) \times \H^{-1/2}(\p \Omega)$ is given as the solution of system~\eqref{sysadj}, which corresponds to the following one:
\begin{equation}
\left\{ \begin{array} {rl}
-\divg(u\nabla p) = -(y-\hat{y}) & \text{in } \Omega, \\[5pt]
 p_{|\p \Omega} = 0 &  \text{on } \p \Omega.
\end{array} \right.
\label{sys-adj}
\end{equation}
Given the definition of~$\mathcal{U}$ in~\eqref{def-UmM}, proving existence and uniqueness of solutions for systems~\eqref{sys-a}-\eqref{sys-adj} is classical (see section~\ref{appendix2} for the analysis).

\subsubsection{On the relaxed outer-problem} \label{sec-num2}
The coefficient $\nu \geq$ was introduced for theoretical purpose. In the numerical experiments we take~$\nu =0$. The relaxed outer-problem~\eqref{pbmainrelaxed} is formulated as follows:
\begin{equation} \label{pbmainrelaxed-calderon} \tag{$\tilde{\ast\ast}$}
\left\{ \begin{array} {l}
\displaystyle \min_{w\in \mathcal{W}} \quad
\frac{1}{2K} \sum_{k=1}^K \| u_k - \hat{u}_k \|_{\L^{\infty}_F(\Omega)}^2
, \\[15pt]
\text{subject to: $G(u_k,w,\hat{y}_k) = 0$ for all $1\leq k\leq K$.}
\end{array} \right.
\end{equation}
We equip~$\L^{\infty}_F(\Omega)$ with the Euclidean norm of $\R^{n_{\mathrm{in}}}$. In order to apply Theorem~\ref{theorem-main} for this particular example, one need to verify that the theoretical assumptions~$(\mathbf{A1})$--$(\mathbf{A3})$ are satisfied. This is done in the Appendix. Assumption~$(\mathbf{A4})$ is satisfied, since $c(y,\hat{y}) = \frac{1}{2} \|y-\hat{y}\|_{\L^2(\Omega)}^2$. Assumption~$(\mathbf{A5})$ is satisfied, as the activation functions chosen for the numerical experiments are $\tanh$ and $\tan^{-1}$, which are of class~$C^2$. We now focus on the numerical results when solving Problem~\eqref{pbmainrelaxed-calderon}, following the method presented in section~\ref{sec-algo}.

\subsubsection{Data generation} In what follows, we consider $\Omega = (0,1)$ (and thus $d=1$), and discretize the different elliptic equations with P1-finite elements on a uniform mesh made of 100 subdivisions. The parameters~$u$ are piecewise constant. Data will be generated with different values of parameters~$u$, choosing as right-hand-sides
\begin{equation*}
f(x) = -u\left((2-\pi^2x^2)\sin(\pi x) +4\pi x\cos(\pi x) \right), \quad
g(x) = 0,
\end{equation*}
such that implicitly the corresponding exact solution is~$y_{\mathrm{ex}}(x) = x^2\sin(\pi x)$.

\subsection{Rediscovering the $L^2$-norm, or not}
Data $\{(\hat{y}_k,\hat{u}_k); 1\leq k \leq K\}$ are generated artificially by solving the inner-problems corresponding to the $L^2$ regularizer $u\mapsto \frac{3}{2}\|u\|_{\mathcal{U}}^2$, with $\mathcal{U} = \R$.  More precisely, given set of coefficient~$u_{\mathrm{truth}}$, the associated states~$y_{\mathrm{truth}} =:\hat{y}_k$ are computed by solving system~\eqref{sys-a}, and next Problem~\eqref{pbstandard-calderon} with~$r(w,u)$ replaced by $u\mapsto \frac{3}{2}\|u\|_{\mathcal{U}}^2$ as regularizer is solved with Algorithm~\ref{algo-inner} that computes the optimal parameters~$\hat{u}_k$.

Next a feed-forward neural network is trained with this data set, using Algorithm~\ref{algo-outer} presented in section~\ref{sec-algo}. The purpose of this test is to compare the trained neural network with the $L^2$ function that represents the ground-truth regularizer to (re-)discover. We designed a feedforward neural network with~$16$~layers, with uniform dimension of the hidden layers~$n_{\ell}=1$, and $x\mapsto \tanh(x)$ as activation function. Convergence of the objective function is presented in Figure~\ref{fig-misfitL2}, with values given in Table~\ref{table-L2}. Figure~\ref{fig-graphL2} represents the graph of the trained ANN compared with the one of the regularizer with which we generated the artificial data, namely $u \mapsto \frac{3}{2}| u|_{\R}^2$.

\begin{minipage}{\linewidth}
	\hspace{-0.00\linewidth}%
	\centering
	\begin{minipage}{0.48\linewidth}
		\begin{figure}[H]
			\includegraphics[trim = 0cm 0cm 0cm 0cm, clip, scale=0.32]{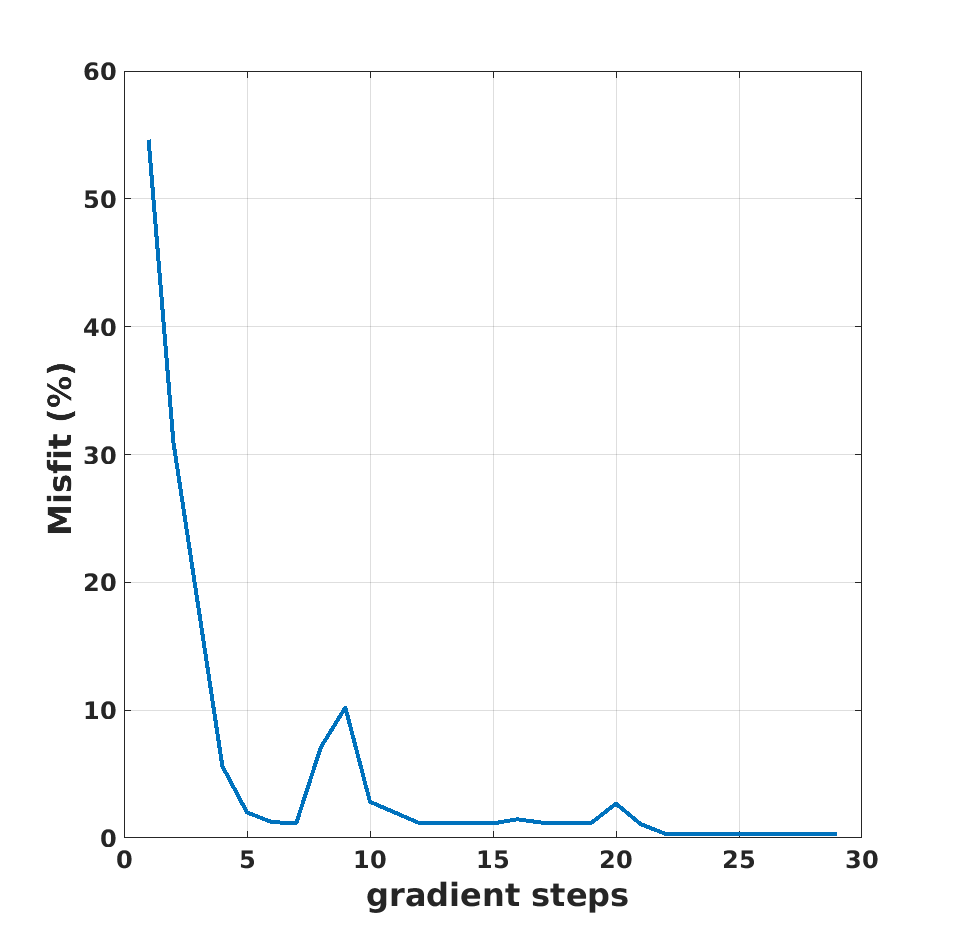}
			\begin{center} linear scale \end{center}
		\end{figure}
	\end{minipage}
	\hspace{0.00\linewidth}
	\begin{minipage}{0.48\linewidth}
		\begin{figure}[H]
			\includegraphics[trim = 0cm 0cm 0cm 0cm, clip, scale=0.32]{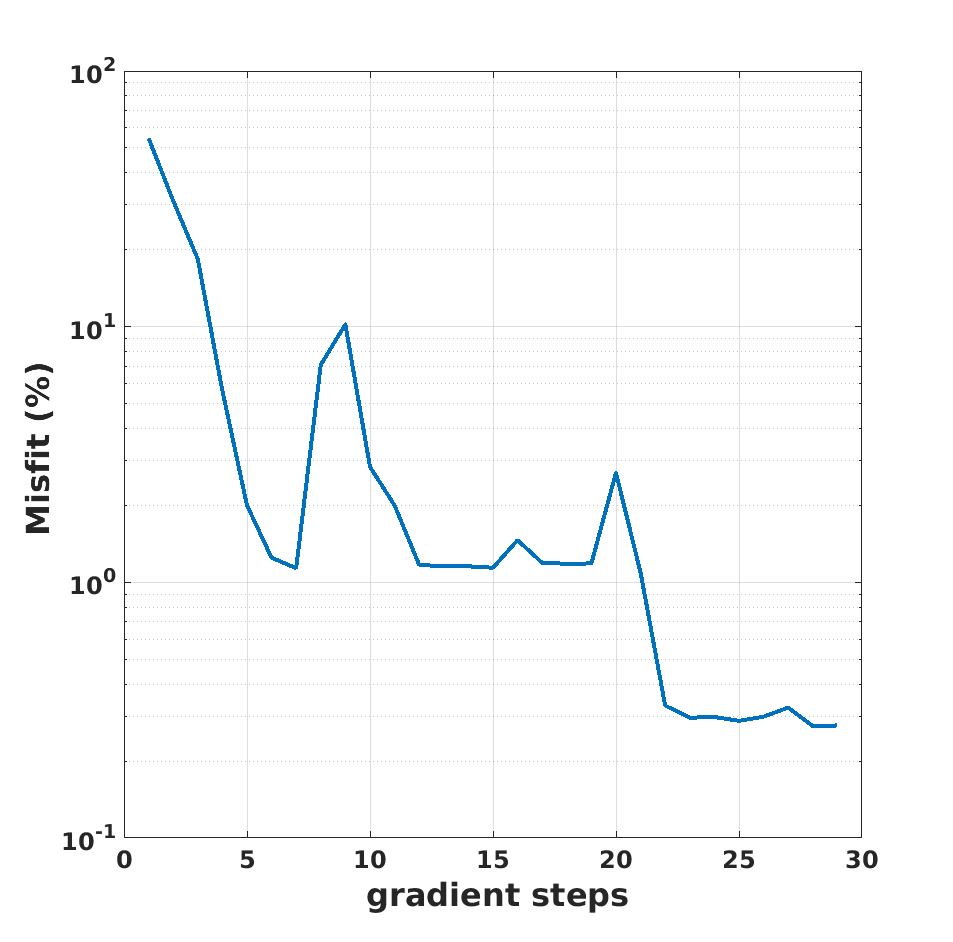}
			\begin{center} log scale \end{center}
		\end{figure}
	\end{minipage}
\\
\vspace*{-15pt}
\begin{figure}[H]
\caption{Misfit evolution during gradient steps for training a neural network from data generated with a $\L^2$-regularizer, in linear scale (left) and log scale (right). \label{fig-misfitL2}}
\end{figure}
\end{minipage}\\

\begin{minipage}{\linewidth}
\begin{figure}[H]
\includegraphics[trim = 0cm 0cm 0cm 0cm, clip, scale=0.32]{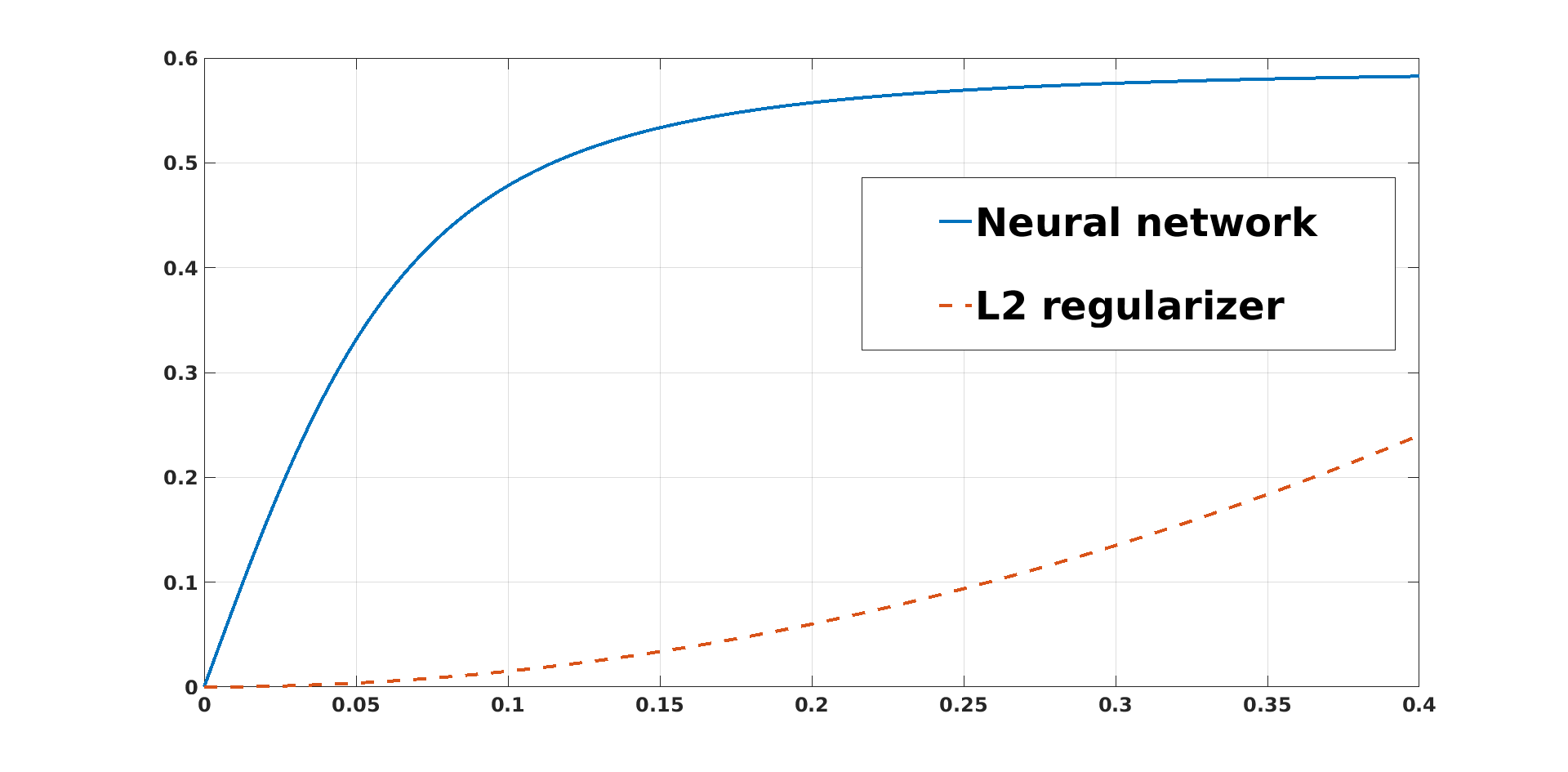}
\end{figure}
\vspace*{-35pt}
\begin{figure}[H]
\caption{Graph of the trained neural network in blue, and of the ground-truth $L^2$ regularizer in red. (For interpretation of the references to color in this figure legend, the reader is referred to the digital version of this article.) \label{fig-graphL2}}
\end{figure}
\end{minipage}\\

In Figure~\ref{fig-misfitL2} and Table~\ref{table-L2} we see that the relative misfit decreased from $54$\% to $0.27$\% after a few iterations. The curves of Figure~\ref{fig-graphL2} suggest non-uniqueness for Problem~\eqref{pbmainrelaxed}, as the residual misfit is of the same range as the approximation error made when discretizing~\eqref{sys-a}. Therefore the trained neural network does the same job as the original $\L^2$-regularizer, but it is clearly different. We note that increasing the number of layers (more specifically with~$L=32$ or~$L=64$) does not change the result, and in particular does not improve the approximation quality.

\begin{table}
\begin{center}
\begin{eqnarray*}
& \begin{array} {|c|c|c|c|c|c|c|c|c|c|c|c|c|c|c|c|}
	\hline
	\text{steps} & 1 & 2 & 3 & 4 & 5 & 6 & 7 & 8 & 9 & 10 & 11 & 12 & 13 & 14 & 15  \\
	\hline 
	\text{misfit (\%)} & 54.6  & 31.1 &    18.4 &    5.62 &    2.00 &    1.25 &    1.14 &    7.09 &    10.2 &    2.83 &    2.00 &    1.17 &    1.16 &    1.15 &    1.14 \\
	\hline
\end{array} & \\
& \hspace*{-28pt} \begin{array} {|c|c|c|c|c|c|c|c|c|c|c|c|c|c|c|}
	\hline 
	\text{steps} & 16 & 17 & 18 & 19 & 20 & 21 & 22 & 23 & 24 & 25 & 26 & 27 & 28 & 29  \\
	\hline
	\text{misfit (\%)} & 1.46 &    1.19 &    1.18 &    1.19 &    2.68 &    1.09 &    0.33 &    0.30 &    0.30 &    0.29 &    0.30 &    0.32 &    0.27 &    0.27  \\
    \hline
\end{array} &
\end{eqnarray*}
\vspace*{-10pt}
		\caption{Misfit vales (in \%) through the different gradient steps.}
		\label{table-L2}
\end{center}
\end{table}

\subsection{Regularizers for compensating noise on the data}

Given a set of $K=10$ piecewise constant coefficients $u_{\mathrm{true}} = (u_{\mathrm{true},k})_{1\leq k \leq K}$, we compute the corresponding states~$y_{\mathrm{true}} = (y_{\mathrm{true},k})_{1\leq k \leq K}$ solutions of~\eqref{sys-a}. We then introduce artificially a noise following a Gaussian distribution that transforms each $y_{\mathrm{true},k}$ into $y_{\mathrm{noisy},k}$. 
The pairs $\{(\hat{y}_k,\hat{u}_k) := 
( y_{\mathrm{noisy},k}, u_{\mathrm{true},k}), 1\leq k \leq K \}$ constitute our data set on which we will train the neural network as regularizer for the inner-problem. The aim of such a regularizer consists in compensating the effects of the noise. Problem~\eqref{pbstandard} equipped with the so-trained regularizer will then enable us to determine the desired coefficients $u$, in spite of the noisy state $\hat{y}$ on which we can rely only.

\begin{minipage}{\linewidth}
	\hspace{-0.00\linewidth}%
	\centering
	\begin{minipage}{0.48\linewidth}
		\begin{figure}[H]
			\includegraphics[trim = 0cm 0cm 0cm 0cm, clip, scale=0.32]{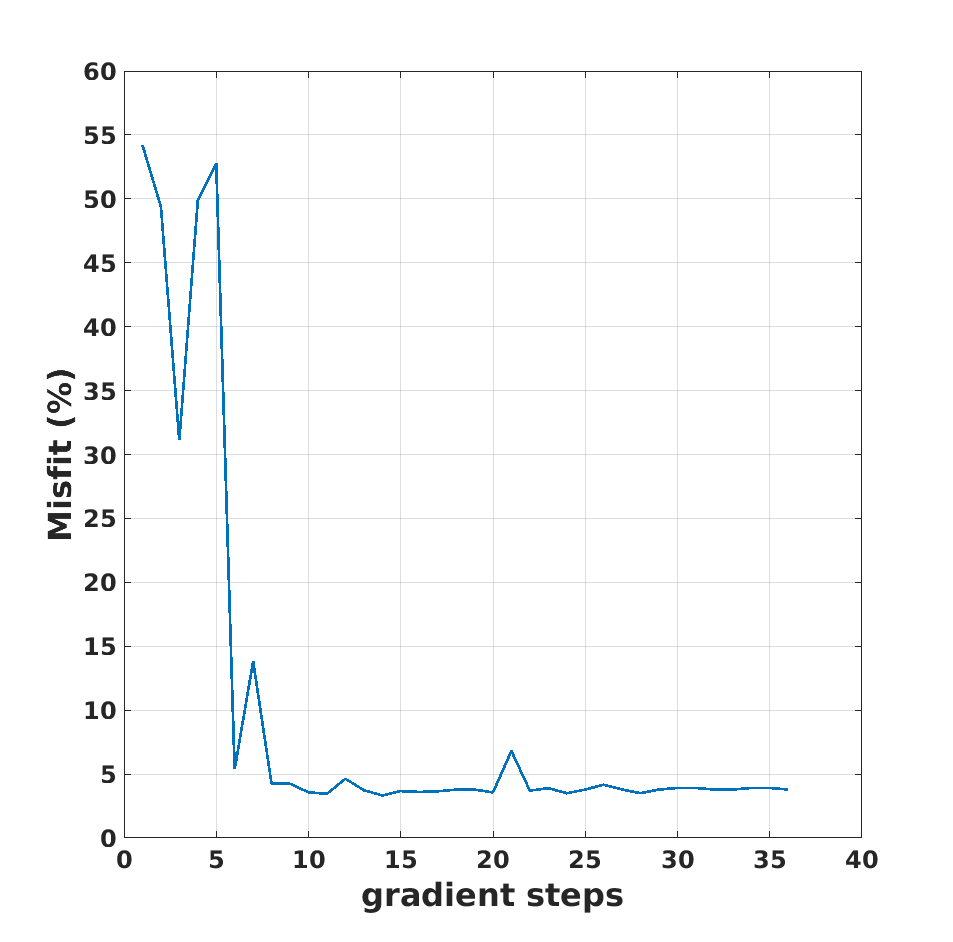}
			\begin{center}\begin{small} linear scale \end{small}\end{center}
		\end{figure}
	\end{minipage}
	\hspace{0.00\linewidth}
	\begin{minipage}{0.48\linewidth}
		\begin{figure}[H]
			\includegraphics[trim = 0cm 0cm 0cm 0cm, clip, scale=0.32]{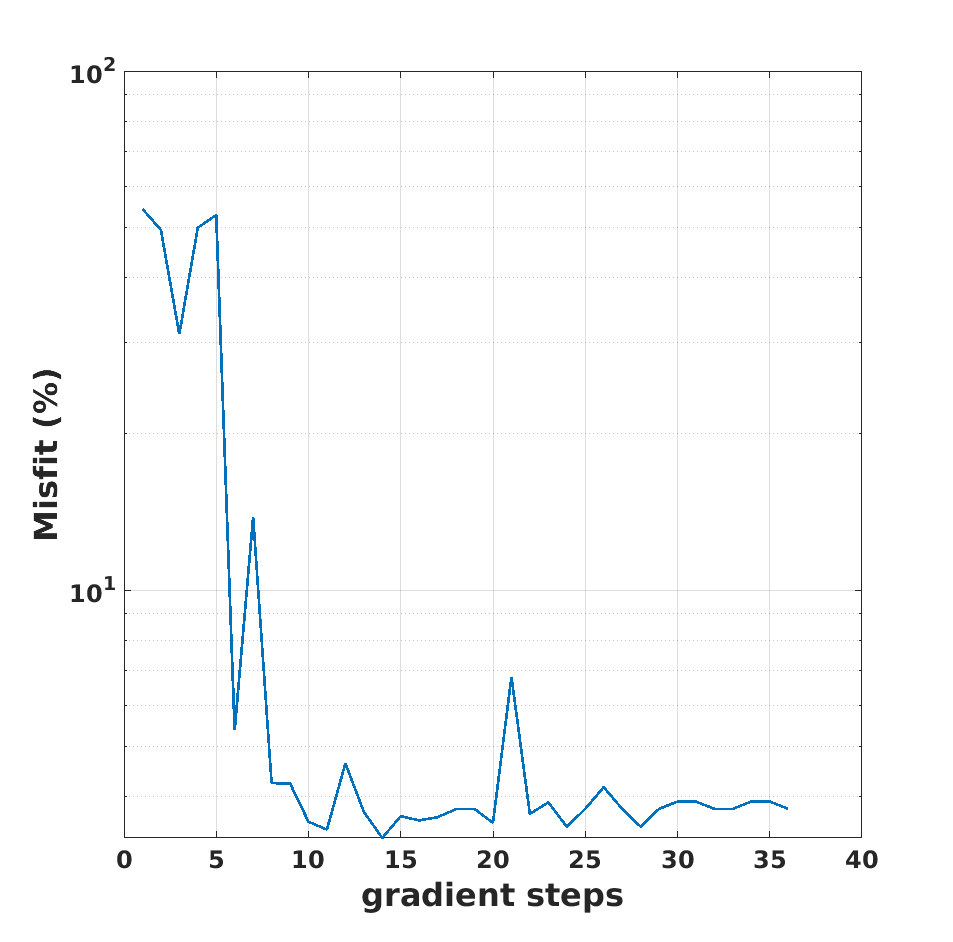}
			\begin{center}\begin{small} log scale \end{small}\end{center}
		\end{figure}
	\end{minipage}
\\
\vspace*{-15pt}
\begin{figure}[H]
\caption{Misfit evolution during gradient steps for training a noise-compensating neural network, in linear scale (left) and log scale (right). \label{fig-unnoise}}
\end{figure}
\end{minipage}\\
\FloatBarrier

\begin{table}
\begin{center}
\begin{scriptsize}
\begin{eqnarray*}
& \begin{array} {|c|c|c|c|c|c|c|c|c|c|c|c|c|c|c|c|c|c|c|}
	\hline
	\text{steps} & 1 & 2 & 3 & 4 & 5 & 6 & 7 & 8 & 9 & 10 & 11 & 12 & 13 & 14 & 15 & 16 & 17 & 18 \\
	\hline 
	\text{misfit (\%)} & 54.2 &   49.4 &   31.1 &   50.0 &   52.8 &   5.37 &   13.9 &   4.25 &   4.24 &   3.58 &   3.45 &   4.63 &   3.74 &   3.33 &   3.67 &   3.60 &   3.65 &   3.78  \\
	\hline
\end{array} & \\
&  \begin{array} {|c|c|c|c|c|c|c|c|c|c|c|c|c|c|c|c|c|c|c|}
	\hline 
	\text{steps} & 19 & 20 & 21 & 22 & 23 & 24 & 25 & 26 & 27 & 28 & 29 & 30 & 31 & 32 & 33 & 34 & 35 & 36  \\
	\hline
	\text{misfit (\%)} &   3.79 &   3.56 &   6.80  &  3.70 &   3.90 &   3.50 &   3.79 &   4.17 &   3.79 &   3.50 &   3.79 &   3.91 &   3.91 &   3.79 &   3.79 &   3.91 &   3.91 &   3.79  \\
    \hline
\end{array} &
\end{eqnarray*}
\end{scriptsize}
\vspace*{-10pt}
		\caption{Misfit vales (in \%) through the different gradient steps.}
		\label{table-unnoise}
\end{center}
\end{table}
\FloatBarrier

\noindent In the illustration presented in Figure~\ref{fig-unnoise} and Table~\ref{table-unnoise}, the initial misfit -- corresponding to a neural network close to zero -- is about 54\%. We trained a feedforward neural network with~$L=8$~layers, and $x\mapsto \tan^{-1}(x)$ as activation function. We see that only few iterations were needed for obtaining a misfit lower than 5\%. The type of neural network we consider does not allow us to reduce the misfit below 3,32\%. Again, increasing the number of hidden layers did not help to improve the result.



\section{Conclusion} \label{sec-conclusion}
We have proposed a method for addressing in practice the data-driven design of regularizers for improving a class of inverse problems formulated as optimal control problems. The class of regularizers is restricted to a set of artificial neural networks, which goes in the sense of the purely nonlinear nature of such a problem. Our approach is bi-level, and based on a relaxation exploiting the necessary optimality conditions for these problems. We saw that such a type of approach enables practical and efficient implementation, providing that the different objective functions and activation functions of the neural network are smooth. Developing a similar approach when the involved quantities are less smooth could present an interesting challenge. Further, investigating applicability of our approach for more complex problems (like image segmentation for example) would demonstrate even more its relevance. In such a context, one may need to take into account neural networks architectures that allow more sophisticated realizations. Finally, noticing that our approach necessitates the evaluation of second-order derivatives of the parameter-to-state mapping for the inner-problem, developing higher order adjoint-based methods for deriving the optimality conditions for the outer-problem could potentially improve the implementation and the applicability to large-scale tasks.

\appendix

\section{Appendix}

In this section we deal with the example considered in section~\ref{sec-num}, namely the problem of identifying a conductivity represented by the variable $u$ in the following elliptic partial differential equation:
\begin{eqnarray}
\left\{ \begin{array}{rl}
-\divg(u\nabla y) = f & \text{in } \Omega, \\
y_{|\p \Omega} = g & \text{on } \p \Omega.
\end{array} \right. \label{sys-a-app}
\end{eqnarray}
Here we set $\mathcal{Y} = \H^1(\Omega)$, $\mathcal{U} = \L^{\infty}(\Omega)$, $\mathcal{Q} = \H^1_0(\Omega) \times \H^{-1/2}(\p \Omega)$ and we recall that  $\mathcal{Q}' = \H^{-1}(\Omega) \times \H^{1/2}(\p \Omega)$. Right-hand-sides are given by $f\in \H^{-1}(\Omega)$ and $g\in \H^{1/2}(\p \Omega)$. In that case, we have
\begin{equation*}
\varphi(y,u) = (-\divg(u\nabla y) -f,\ y_{|\p \Omega}-g) \in \mathcal{Q}'.
\end{equation*}
It is clear that~$\varphi$ is of class~$C^2$ over~$\mathcal{Y} \times \mathcal{U}$, and therefore Assumption~$(\mathbf{A1})$ is satisfied. We have $\varphi'_y(y,u).\tilde{y} =  (-\divg(u\nabla \tilde{y}) ,\ \tilde{y}_{|\p \Omega})$, and it is well-known that the linear mapping~$\varphi'_y(y,u)$ is surjective from~$\mathcal{Y}\times\mathcal{U}$ onto~$\mathcal{Q}'$ (see section~\ref{appendix2} for more details). Therefore Assumption~$(\mathbf{A2})$ is satisfied too. We prove in Lemma~\ref{lemma-A3-app} that Assumption~$(\mathbf{A3})$ is satisfied too. We have already explained in section~\ref{sec-num2} that Assumptions $(\mathbf{A4})$-$(\mathbf{A5})$ are satisfied.

\subsection{Results related to the inverse conductivity problem} \label{appendix2}

Recall that $\H^{-1}(\Omega) = \left(\H^1_0(\Omega)\right)'$, and that $0<m<M$ defines~$\mathcal{U}_{m,M}$ as in~\eqref{def-UmM}. It is well-known that for all $u\in \mathcal{U}_{m,M}$, $f\in \H^{-1}(\Omega)$ and $g \in \H^{1/2}(\p \Omega)$ system~\eqref{sys-a-app} admits a unique solution $y \in \H^1(\Omega)$. This is a consequence of the Lax-Milgram theorem. Moreover, there exists a constant $C>0$ depending only on $\Omega$ such that the following estimate holds:
\begin{equation}
\|y\|_{\H^1(\Omega)} \leq \frac{1}{m}\left(\|f\|_{\H^{-1}(\Omega)} + M\|g\|_{\H^{1/2}(\p \Omega)}\right).
\label{estellip}
\end{equation}
Furthermore, the following result show that the parameter-to-state mapping $\mathbb{S}$ is Lipschitz:
\begin{lemma}
Let be $M>m>0$ and assume that $u_1, u_2 \in \mathcal{U}_{m,M}$. Then the solutions $y_1$ and $y_2$ of system~\eqref{sys-a} corresponding to $u =u_1$ and $u=u_2$ respectively satisfy
\begin{equation}
\|y_1-y_2\|_{\H_0^1(\Omega)} \leq \frac{1}{m^2} \|u_1-u_2\|_{\L^{\infty}(\Omega)}
\left(\|f\|_{\H^{-1}(\Omega)} + M\|g\|_{\H^{1/2}(\p \Omega)} \right), \label{lipest}
\end{equation}
where we set $\|y\|_{\H_0^1(\Omega)}:= \|\nabla y\|_{\L^2(\Omega)}$.
\end{lemma}

\begin{proof}
Substituting the first equations of~\eqref{sys-a} satisfied by $y_1$ and $y_2$ leads us to
\begin{equation*}
-\divg(u_1\nabla y_1) + \divg(u_2\nabla y_2)=0 \Rightarrow
-\divg(u_1\nabla (y_1-y_2)) + \divg((u_2-u_1)\nabla y_2)=0.
\end{equation*}
Note that $y_1-y_2 \in \H^1_0(\Omega)$. Then, taking the inner product of the last identity above by $y_1-y_2$ and integrating by parts yields
\begin{eqnarray*}
& & \int_{\Omega}u_1 |y_1-y_2|^2\, \d \Omega = \int_{\Omega} (u_2-u_1) \nabla y_2 \cdot \nabla (y_1-y_2) \, \d \Omega \\
&  \Rightarrow & m \|\nabla(y_1-y_2)\|^2_{\L^2(\Omega)} \leq 
\|u_1 - u_2\|_{\L^{\infty}(\Omega)} \|\nabla y_2\|_{\L^2(\Omega)} 
\|\nabla (y_1- y_2)\|_{\L^2(\Omega)}\\
&  \Rightarrow & m \|y_1-y_2\|_{\H^1_0(\Omega)} \leq 
\|u_1 - u_2\|_{\L^{\infty}(\Omega)} \|y_2\|_{\H^1_0(\Omega)}.
\end{eqnarray*}
Using estimate~\eqref{estellip} satisfied by $y_2$, the previous one implies
\begin{equation*}
m \|y_1-y_2\|_{\H^1_0(\Omega)} \leq 
\|u_1 - u_2\|_{\L^{\infty}(\Omega)} \frac{1}{m}\left(\|f\|_{\H^{-1}(\Omega)} + M\|g\|_{\H^{1/2}(\p \Omega)}\right).
\end{equation*}
Thus estimate~\eqref{lipest} follows and completes the proof.
\end{proof}

\subsubsection*{On the inner-problem}
Let be $w\in \mathcal{W}$. Considering $\hat{y} \in \L^2(\Omega)$, the inner-problem~\eqref{pbstandard} is formulated as follows:
\begin{equation}
\left\{ \begin{array} {l}
\displaystyle \min_{(y,u)\in \H^1(\Omega) \times \mathcal{U}_{m,M}} 
\left(
J(y,u,w) = \frac{1}{2}\| y-\hat{y}\|_{\L^2(\Omega)}^2 + \gamma\circ r(w,u) 
\right)\\[10pt]
\text{subject to \eqref{sys-a}.}
\end{array} \right.
\label{pbstandardex}
\end{equation}
The result below establishes existence of minimizers for Problem~\eqref{pbstandardex}, and thus Assumption~$(\mathbf{A3})$ is fulfilled.

\begin{lemma} \label{lemma-A3-app}
Given $w\in \mathcal{W}$, Problem~\eqref{pbstandardex} admits global minimizers.
\end{lemma}

\begin{proof}
Since the function $J$ of Problem~\eqref{pbstandardex} is bounded from below, and since system~\eqref{sys-a} is well-posed, $J$ admits a minimizing sequence~$(y_n,u_n)_n$ of feasible solutions. From the definition of~$\mathcal{U}_{m,M}$, the sequence~$(u_n)_n$ is uniformly bounded in $\L^{\infty}_F(\Omega)$, which is of finite dimension. Therefore, up to extraction, the sequence~$(u_n)_n$ converges strongly in $\L^{\infty}(\Omega)$ towards $u\in \mathcal{U}_{m,M}$. From estimate~\eqref{estellip}, the sequence~$(y_n)_n$ is bounded in~$\H^1(\Omega)$. In virtue of the Banach-Alaoglu theorem, up to extraction it converges weakly towards~$y \in \H^1(\Omega)$. Let us show that~$y$ in solution of system~\eqref{sys-a} corresponding to~$u$. For all $\varphi \in \H^1_0(\Omega)$, using the Green formula, we have
\begin{eqnarray*}
\langle - \divg(u\nabla y) - f ; \varphi \rangle_{\H^{-1}(\Omega),\H^1_0(\Omega)} & = & 
\langle - \divg(u\nabla y) +  \divg(u_n\nabla y_n) ; \varphi \rangle_{\H^{-1}(\Omega),\H^1_0(\Omega)} \\
& = & \langle u\nabla y - u_n\nabla y_n ; \nabla \varphi \rangle_{[\L^2(\Omega)]^d} \\
& = & \langle (u-u_n)\nabla y_n; \nabla \varphi \rangle_{[\L^2(\Omega)]^d}
+ \langle u\nabla(y-y_n); \nabla \varphi \rangle_{[\L^2(\Omega)]^d}\\
& = & \langle (u-u_n)\nabla y_n; \nabla \varphi \rangle_{[\L^2(\Omega)]^d}
- \langle \divg(u\nabla \varphi);y-y_n \rangle_{\H^{-1}(\Omega),\H^1_0(\Omega)} ,
\end{eqnarray*}
where we have used that $y-y_n \in \H^1_0(\Omega)$. This yields
\begin{eqnarray*}
\left|\langle - \divg(u\nabla y) - f ; \varphi \rangle_{\H^{-1}(\Omega),\H^1_0(\Omega)}\right| & \leq & 
\|u-u_n\|_{\L^{\infty}(\Omega)} \|\nabla y_n\|_{[\L^2(\Omega)]^d}
\|\nabla \varphi \|_{[\L^2(\Omega)]^d} \\
& & + 
\left|\langle \divg(u\nabla \varphi) ; y-y_n  \rangle_{\H^{-1}(\Omega),\H^1_0(\Omega)}\right|.
\end{eqnarray*}
Passing to the limit, we deduce that $-\divg(u\nabla y) = f$ in $\H^{-1}(\Omega)$. Passing to the limit in the equality ${y_n}_{|\p \Omega} = g$ is straightforward. Thus $(y,u)$ is a solution of Problem~\eqref{pbstandardex}, which completes the proof.
\end{proof}

\subsection{On the optimal control approach for solving an inverse problem}\label{appendix1}

Let us fix the ANN weights~$w$ for this subsection. A classical approach for solving inverse problems consists in taking into account the equality constraint $\varphi(y,u) = 0$ by incorporating it in the functional to minimize, as follows:
\begin{eqnarray*}
 \min_{u\in \mathcal{U}, y \in \mathcal{Y}} \Big(J(y,u,w) := 
\alpha \| \varphi(y,u) \|_{\mathcal{Q}}^2 
+ c(y,\hat{y}) + \gamma \circ r(w,u)\Big),
\end{eqnarray*}
where $\alpha >0$ is a given parameter (and $w$ too). The contribution of the state equation~\eqref{sys-a-app} in the functional to minimize is
\begin{equation*}
\|\varphi(y,u)\|_{\mathcal{Q}}^2 = 
\| -\divg(u\nabla y) -f \|_{\H^{-1}(\Omega)}^2 
+ \|y_{|\p \Omega} -g\|_{\H^{1/2}(\p \Omega)}^2. 
\end{equation*}
This approach would require differentiation of the data (represented here by~$u$), more specifically the computation of~$\nabla u$, and thus could lead to approximation biases, especially when the set of data is incomplete. Instead of that, we consider the state equation as a constraint, that we impose by introducing a Lagrange multiplier denoted by $(p,q) \in \H^1_0(\Omega) \times \H^{-1/2}(\p \Omega) = \mathcal{Q} \simeq \mathcal{Q}''$, namely the {\it adjoint} variable, leading to the search of a saddle-point of a Lagrangian functional, given as:
\begin{equation*}
\mathscr{L}(y,u,p)  =  c(y,\hat{y})  + \gamma \circ r(w,u) + \langle (p,q)\, , \varphi (y,u) \rangle_{\mathcal{Q},\mathcal{Q}'}.
\end{equation*}
A saddle-point of $\mathscr{L}$ is obtained as a critical point with respect to the variables $(y,u)$ and $(p,q)$. The duality product $\langle (p,q)\, , \varphi (y,u) \rangle_{\mathcal{Q},\mathcal{Q}'}$ corresponds to a weak variational formulation, namely
\begin{eqnarray*}
\langle (p,q)\, , \varphi (y,u) \rangle_{\mathcal{Q},\mathcal{Q}'} & = &
\left\langle p\, , -\divg ( u \nabla y) -f \right\rangle_{\H^1_0(\Omega),\H^{-1}(\Omega)}
+ \left\langle q\, , y_{|\p \Omega} -g \right\rangle_{\H^{1/2}(\p \Omega),\H^{-1/2}(\p \Omega)}
\\
& = &  \int_{\Omega} \nabla p : (u\nabla y) \, \d \Omega
-\left\langle p\, , f \right\rangle_{\H^1_0(\Omega),\H^{-1}(\Omega)}
+ \left\langle q\, , y_{|\p \Omega} -g \right\rangle_{\H^{1/2}(\p \Omega),\H^{-1/2}(\p \Omega)},
\end{eqnarray*}
where we have used the Green formula. Thus we see that the data represented by the function $u$ does not need to be differentiated.



\printbibliography

\end{document}